\documentclass[12pt]{elsarticle}

\journal{} 

\usepackage{amssymb,amscd,amsthm,amsxtra}
\usepackage{dsfont,latexsym}
\usepackage{url}

\usepackage{mathrsfs}

\usepackage{color}
\definecolor{citation}{rgb}{0.2,0.6,0.2}
\definecolor{formula}{rgb}{0.1,0.2,0.6}
\definecolor{url}{rgb}{0,0,0.4}
\usepackage[colorlinks=true,linkcolor=formula,urlcolor=url,citecolor=citation]{hyperref}

\usepackage{everysel}
\EverySelectfont{%
\fontdimen2\font=0.4em
\fontdimen3\font=0.2em
\fontdimen4\font=0.1em
\fontdimen7\font=0.1em
\hyphenchar\font=`\-
}

\newtheorem{thm}{Theorem}[section]
\newtheorem{corollary}[thm]{Corollary}
\newtheorem{lemma}[thm]{Lemma}
\newtheorem{prop}[thm]{Proposition}
\theoremstyle{definition}

\theoremstyle{remark}
\newtheorem{rem}[thm]{Remark}
\newtheorem{exe}[thm]{Example}
\numberwithin{equation}{section}

\newcommand{\CC}{\mathscr{C}}
\newcommand{\FF}{\mathscr{F}}

\newcommand{\LL}{\mathscr{L}}

\newcommand{\SF}{{\mathscr S}}
\newcommand{\TT}{\mathscr{T}}

\newcommand{\R}{{\mathds R}}
\newcommand{\N}{{\mathds N}}
\newcommand{\Z}{{\mathds Z}}

\def\dys{\displaystyle}
\def\eps{\varepsilon}
\def\ds{(-\Delta)^s}

\newlength{\defbaselineskip}
\newcommand{\setlinespacing}[1]
           {\setlength{\baselineskip}{#1 \defbaselineskip}}

\begin{document}

        
\begin{frontmatter}

\title{Hitchhiker's guide \\ to the fractional Sobolev spaces}

\author[torv]{Eleonora Di Nezza}
\ead{dinezza@mat.uniroma2.it}

\author[torv,parma]{Giampiero Palatucci\fnref{fnGP}\corref{cor1}}\ead{giampiero.palatucci@unimes.fr}

\author[torv,milano]{Enrico Valdinoci\fnref{thanksEV}}
\ead{enrico@math.utexas.edu}

\cortext[cor1]{Corresponding author}

\fntext[fnGP]{Giampiero Palatucci has been supported by Istituto Nazionale di Alta Matematica ``F.~\mbox{Severi}''~(Indam)
 and by ERC grant 207573 ``Vectorial problems''.
}

\fntext[thanksEV]{Enrico Valdinoci has been supported by the ERC grant ``$\epsilon$ Elliptic Pde's and Symmetry of Interfaces and Layers for Odd Nonlinearities'' and the FIRB project~``A\&B~Analysis and Beyond''.}

\fntext[thanksRM2]{These notes grew out of a few lectures given in an undergraduate class held at the Universit\`a di Roma ``Tor Vergata''. It is a pleasure to thank the students for their warm interest, their sharp observations and their precious feedback.}

\address[torv]
{Dipartimento di Matematica, Universit\`a di Roma ``Tor Vergata'' - Via della Ricerca Scientifica, 1, 00133 Roma, Italy}
\address[parma]
{Dipartimento di Matematica, Universit\`a degli Studi di Parma, Campus - Viale delle Scienze,~53/A, 43124 Parma, Italy}
\address[milano]
{Dipartimento di Matematica, Universit\`a degli Studi di Milano - Via Saldini,~50, 20133 Milano, Italy}

\begin{abstract}
This paper deals with the fractional Sobolev spaces $W^{s,p}$.
We analyze the relations among some of their possible definitions and their role in the trace theory. We prove continuous and compact embeddings, investigating the problem of the extension domains and other regularity results.

Most of the results we present here are probably well known to the experts, but we believe that our proofs are original and we do not make use of any interpolation techniques nor pass through the theory of Besov spaces.
We also present some counterexamples in non-Lipschitz domains.

\end{abstract}


\end{frontmatter}


\setcounter{tocdepth}{1}
{\small \tableofcontents}

\section{Introduction}\label{sec_intro}

These pages are for students and young researchers
of all ages who may like to hitchhike their way
from~$1$ to~$s\in(0,1)$. To wit, for anybody who,
only endowed with some basic undergraduate
analysis course (and knowing
{\tt where his towel is}),
would like to pick up some quick, crash and essentially
self-contained
information on the fractional Sobolev spaces~$W^{s,p}$.
\vspace{2mm}

The reasons for such a hitchhiker to start this
adventurous trip might be of different kind: (s)he could
be driven by mathematical curiosity, or could
be tempted by the many applications that
fractional calculus seems to have recently experienced.
In a sense, fractional Sobolev spaces
have been a classical topic in functional
and harmonic analysis all along, and some important books,
such as~\cite{Lan73,Ste70} treat the topic
in detail. On the other hand,
fractional spaces, and the corresponding nonlocal
equations, are now experiencing impressive
applications in different subjects, such as, among others,
the thin obstacle problem~\cite{Sil07, MS08},
optimization~\cite{DL76},
finance~\cite{CT04}, phase transitions~\cite{ABS98, CSM05, SiV09, FV11, GP06}, 
stratified materials~\cite{SV09, CV09, CV10},
anomalous diffusion~\cite{MK00, Vaz10, MMM11}, 
crystal dislocation~\cite{Tol97, GM10, BKM10},
soft thin films~\cite{Kur06},
semipermeable membranes
and flame propagation~\cite{CMS11},
conservation laws~\cite{conla},
ultra-relativistic
limits of quantum mechanics~\cite{FeffermanL86},
quasi-geostrophic flows~\cite{MajdaT96,cordoba98, CVa10},
multiple scattering~\cite{DuistermaatG75, CK98, SCSC},
minimal surfaces~\cite{CRS10, CV11},
materials science~\cite{MSBB},
water waves~\cite{Sto57,Z68,Whi74,Craig92,CraigG94,
Naumkin94,CW-95,Pan96,Craig97,Craig00,WW,HN05,Taber08, LlaveV}, elliptic problems with measure data~\cite{Min07, KPU11}, non-uniformly elliptic problems~\cite{ELM04}, gradient potential theory~\cite{Min11} and~singular set of minima of variational functionals~\cite{Min03, Min06}.
{\tt Don't panic}, instead, 
see also~\cite{Sil05,Sil07} for further motivation.
\vspace{2mm}

For these reasons, we thought that it could be of some interest to write down these notes --
or, more frankly, we wrote them
just because {\tt if you really want to
understand something, the best way is to
try and explain it to someone else}.
\vspace{2mm}

Some words may be needed to clarify the style of these
pages have been gathered.
We made the effort of making a rigorous
exposition, starting from scratch,
trying to use the least amount of
technology and with the simplest, low-profile language
we could use -- since {\tt capital letters
were always the best way of dealing with things
you didn't have a good answer to}. 
\vspace{2mm}

Differently from
many other references, we make no use of
Besov spaces\footnote{About this, we would like
to quote~\cite{Joh85}, according to which
``The paradox of Besov spaces is that the very thing that makes them so
successful also makes them very difficult to present and to learn''.}
or interpolation techniques,
in order to make the arguments as elementary as possible
and the exposition suitable for everybody, since 
{\tt when you are a student or whatever, and you can't afford
a car, or a plane fare, or even a train fare, all you can
do is hope that someone will stop and pick you up},
and {\tt it's nice to think that one could, even here and now,
be whisked away just by hitchhiking}.\vspace{2mm}

Of course, by dropping fine technologies and
powerful tools, we will miss
several very important features,
and we apologize for this. So, we highly recommend
all the excellent, classical books
on the topic, such as~\cite{Lan73, Ste70, Ada75,
Tri77, Tri78, Zie89, Run96,  Tar07, MS09, Leo09}, 
and the many references given therein. Without them,
our reader would remain just a hitchhiker, losing the opportunity of performing
the next crucial step towards a
full mastering of the subject and becoming the captain of a spaceship.
\vspace{2mm}

In fact, compared to other
Guides, this one is not {\tt definitive},
and it {\tt is a very evenly edited book
and contains many passages that simply seemed
to its editors a good idea at the time}. In any case, of course,
we know that
we cannot {\tt solve any major problems just with potatoes} --
it's fun to try and see how far one can get though.

In this sense, while
most of the results we present here are probably well known to the experts, we believe
that the exposition is somewhat original.\vspace{2mm}

These are the topics we cover.
In Section~\ref{definizioni},
we define the fractional Sobolev spaces~$W^{s,p}$
via the Gagliardo approach and we investigate some of their basic properties. In Section~\ref{sec_laplacian} 
we focus on the Hilbert case~$p=2$, dealing
with its relation with the fractional Laplacian,
and letting the principal value integral definition
interplay with the definition in the Fourier space. Then, in Section~\ref{sec_cost} we analyze the asymptotic behavior of the constant factor that appears in the definition of the fractional Laplacian.

Section~\ref{sec_estensione}
is devoted to the extension problem of a function in~$W^{s,p}(\Omega)$
to $W^{s,p}(\R^n)$: technically, this is slightly
more complicated than the classical analogue for integer Sobolev
spaces, since the extension interacts with the values taken by
the function in~$\Omega$ via the Gagliardo norm 
and the computations have to take care of it.

Sobolev inequalities and continuous embeddings are dealt with
in Section~\ref{sec_sobolev},
while Section~\ref{sec_compactness}
is devoted to compact embeddings.
Then, in Section~\ref{sec_holder}, we point out that functions
in~$W^{s,p}$ are continuous when~$sp$ is large enough.

In Section~\ref{sec_example}, we present some counterexamples in non-Lipschitz domains.

\vspace{2mm}

After that, we hope that
our hitchhiker reader has enjoyed his trip
from the integer Sobolev spaces to the fractional ones,
with the {\tt advantages of being able to get more quickly
from one place to another - particularly when
the place you arrived at had probably become, as a result of this,
very similar to the place you had left}.
\vspace{2mm}

The above sentences written in {\tt old-fashioned fonts}
are Douglas Adams's 
of course, and we took the latitude
of adapting their meanings to our purposes.
The rest of these pages are written in a more
conventional, may be boring, but hopefully rigorous, style.

\vspace{3mm}
\section{The fractional Sobolev space $W^{s,p}$}\label{definizioni}
This section is devoted to the definition of the fractional Sobolev spaces.

No prerequisite is needed. We just recall the definition of the Fourier transform of a distribution.
First, consider the Schwartz space $\SF$ of rapidly decaying $C^{\infty}$ functions in $\R^n$. The topology of this space is generated by the seminorms
$$
p_N(\varphi)= \sup_{x\in\R^n} (1+|x|)^N \sum_{|\alpha|\leq N} |D^{\alpha}\varphi(x)|\,, \quad N=0,1,2,...\,, 
$$
where $\varphi\in \SF(\R^n)$. 
Let $\SF'(\R^n)$ be the set of all tempered distributions, that is the topological dual of $\SF(\R^n)$. As usual, for any $\varphi\in \SF(\R^n)$, we denote by
$$
\FF\varphi(\xi)\, = \, \frac{1}{(2\pi)^{{n}/{2}}} \! \int_{\R^n} e^{-i\xi \cdot x}\,\varphi(x)\,dx
$$
the Fourier transform of $\varphi$ and we recall that one can extend $\FF$ from $\SF(\R^n)$ to $\SF'(\R^n)$.

\vspace{2mm}

Let $\Omega$ be a general, possibly non smooth, open set in $\R^n$. For any real $s>0$ and for any $p\in [1,\infty)$, 
we want to define the fractional Sobolev spaces $W^{s,p}(\Omega)$. In the literature, fractional Sobolev-type spaces are also called {\it Aronszajn}, {\it Gagliardo} or {\it Slobodeckij spaces}, by the name of the ones who introduced them, almost simultaneously (see~\cite{Aro55, Gag58, Slo58}).

\vspace{2mm}

We start by fixing the fractional exponent $s$ in $(0,1)$. For any $p \in[1,+\infty)$, we define $W^{s,p}(\Omega)$ as follows
\begin{equation}\label{def1}
W^{s,p}(\Omega):= \left\{ u\in L^p(\Omega)\; :\; \frac{|u(x)-u(y)|}{|x-y|^{\frac{n}{p}+s}} \in L^p(\Omega \times \Omega)  \right\};
\end{equation}
i.e, an intermediary Banach space between $L^p(\Omega)$ and $W^{1,p}(\Omega)$, endowed with the natural norm
\begin{equation}\label{def2}
\|u\|_{W^{s,p}(\Omega)}:= \left( \int_{\Omega}|u|^p \,dx\,+\,\int_{\Omega} \int_{\Omega} \frac{|u(x)-u(y)|^p}{|x-y|^{n+sp}} dx\, dy  \right)^{\frac{1}{p}}\!,
\end{equation}
where the term
$$
\dys 
[u]_{W^{s,p}(\Omega)} := \left( \int_{\Omega} \int_{\Omega} \frac{|u(x)-u(y)|^p}{|x-y|^{n+sp}} dx\, dy  \right)^{\frac{1}{p}}$$
is the so-called {\it Gagliardo {\rm(}semi{\,\rm)}norm} of $u$.
\vspace{2mm}

It is worth noticing that, as in the classical case with $s$ being an integer, the space $W^{s',p}$ is continuously embedded in $W^{s,p}$ when $s\leq s'$, as next result points out.

\begin{prop}\label{enrico}
Let $p\in[1,+\infty)$ and $0<s\leq s'<1$. Let $\Omega$ be an open set in~$\R^n$ and $u:\Omega\rightarrow \R$ be a measurable function. Then 
$$
\|u\|_{W^{s,p}(\Omega)} \leq C \|u\|_{W^{s',p}(\Omega)}
$$
for some suitable positive constant $C=C(n,s,p)\geq 1$. In particular, $$W^{s',p}(\Omega)\subseteq W^{s,p}(\Omega)\,.$$
\end{prop}

\begin{proof}
First,
\begin{eqnarray*}
\int_{\Omega}\int_{\Omega\, \cap \, \{|x-y|\geq 1\}} \frac{|u(x)|^p}{|x-y|^{n+sp}} \,dx\,dy \! &\leq & \! \int_{\Omega} \left( \int_{|z|\geq 1} \frac{1}{|z|^{n+sp}}\,dz  \right) |u(x)|^p\,dx\\
\\
&\leq & \! C(n,s,p) \|u\|^p_{L^p(\Omega)}\,,
\end{eqnarray*}
where we used the fact that the kernel $1/|z|^{n+sp}$ is integrable since $n+sp>n$.

\vspace{1mm}

Taking into account the above estimate, it follows
\begin{eqnarray}\label{enrico1}
&& \int_{\Omega}\int_{\Omega\,\cap\,\{|x-y|\geq 1\}}\!\!\frac{|u(x)-u(y)|^p}{|x-y|^{n+sp}} \,dx\,dy \nonumber \\
\nonumber \\
&& \qquad \qquad \qquad \qquad \leq 2^{p-1}\!\!\int_{\Omega}\int_{\Omega\,\cap\{|x-y|\geq 1\}}\!\!\frac{|u(x)|^p+|u(y)|^p}{|x-y|^{n+sp}} \,dx\,dy \nonumber \\
\nonumber \\
&&  \qquad \qquad \qquad \qquad \leq 2^{p} C(n,s,p) \|u\|^p_{L^p(\Omega)}\,.
\end{eqnarray}

On the other hand,
\begin{equation}\label{enrico2}
\int_{\Omega}\int_{\Omega\,\cap\,\{|x-y|<1\}} \frac{|u(x)-u(y)|^p}{|x-y|^{n+sp}} \,dx\,dy \, \leq\,  \int_{\Omega}\int_{\Omega\,\cap\,\{|x-y|< 1\}}\frac{|u(x)-u(y)|^p}{|x-y|^{n+s'p}} \,dx\,dy\,.
\end{equation}

\vspace{1mm}

Thus, combining~\eqref{enrico1} with~\eqref{enrico2}, we get
\begin{eqnarray*}
&& \int_{\Omega}\int_{\Omega} \frac{|u(x)-u(y)|^p}{|x-y|^{n+sp}} \,dx\,dy \nonumber \\
 &&\qquad \qquad \qquad  \leq 2^{p} C(n,s,p) \|u\|^p_{L^p(\Omega)}+ \int_{\Omega}\int_{\Omega} \frac{|u(x)-u(y)|^p}{|x-y|^{n+s'p}} \,dx\,dy
\end{eqnarray*}
and so
\begin{eqnarray*}
\|u\|^p_{W^{s,p}(\Omega)} \! & \leq & \! \big(2^p C(n,s,p)+1\big)\|u\|^p_{L^p(\Omega)}
+ \int_{\Omega}\int_{\Omega} \frac{|u(x)-u(y)|^p}{|x-y|^{n+s'p}} \,dx\,dy \\
\\
&\leq &\! C(n,s,p) \|u\|^p_{W^{s',p}(\Omega)},
\end{eqnarray*}
which gives the desired estimate, up to relabeling the constant $C(n,p,s)$.
\end{proof}
\vspace{1mm}

We will show in the forthcoming Proposition~\ref{enrico23} that the result in Proposition~\ref{enrico} holds also in the limit case, namely when $s'=1$, but for this we have to take into account the regularity of ~$\partial\Omega$  (see Example~\ref{exa_lip}).

\vspace{1mm}
 
As usual, for any $k\in \N$ and $\alpha\in (0,1]$, we say that $\Omega$ is of class $C^{k,\alpha}$ if there exists $M>0$ such that for any $x\in \partial\Omega$ there exists a ball $B=B_r(x)$, $r>0$, and an isomorphism $ T: Q\to B$\label{formula_t}
such that
$$
T \in C^{k,\alpha}(\overline{Q}), \ \
T^{-1}\in C^{k,\alpha}(\overline{B}), \ \
T(Q_+)=B \cap \Omega, \ \ T(Q_0)=B \cap \partial\Omega
$$
$$
\text{and} \ \ \|T\|_{C^{k,\alpha}(\overline{Q})}+\|T^{-1}\|_{C^{k,\alpha}(\overline{B})}\,\leq \,M,
$$
where
$$
Q:=\left\{ x=(x',x_n)\in\R^{n-1}\times \R \,:\, |x'|<1\;\mbox{and}\;|x_n|<1 \right\},$$
$$
Q_{+}:=\left\{ x=(x',x_n)\in\R^{n-1}\times \R \,:\, |x'|<1\;\mbox{and}\;0<x_n<1 \right\}
$$
$$
\text{and} \ \  Q_0:=\left\{ x\in Q\,:\, x_n=0\right\}.
$$
\vspace{1mm}

 We have the following result.
\begin{prop}\label{enrico23}
Let $p\in[1,+\infty)$ and $s\in(0,1)$. Let $\Omega$ be an open set in $\R^n$ of class $C^{0,1}$ with bounded boundary and $u:\Omega\rightarrow \R$ be a measurable function. Then 
\begin{equation}\label{enrico33}
\|u\|_{W^{s,p}(\Omega)} \leq C \|u\|_{W^{1,p}(\Omega)}
\end{equation}
for some suitable positive constant $C=C(n,s,p)\geq 1$. In particular, $$W^{1,p}(\Omega)\subseteq W^{s,p}(\Omega)\,.$$
\end{prop} 

\begin{proof}
Let $u\in W^{1,p}(\Omega)$. Thanks to the regularity assumptions on the domain~$\Omega$, we can extend $u$ to a function $\tilde{u}:\R^n\rightarrow \R$ such that $\tilde{u}\in W^{1,p}(\R^n)$ and~$\|\tilde{u}\|_{W^{1,p}(\R^n)}\leq C\|u\|_{W^{1,p}(\Omega)}$ for a suitable constant~$C$~(see, e.g.,~\cite[Theorem 7.25]{GT01}).

 Now, using the change of variable $z=y-x$ and the H\"{o}lder inequality, we have
\begin{eqnarray}\label{enrico26}
&& \int_{\Omega}\int_{\Omega\,\cap\,\{|x-y|<1\}} \frac{|u(x)-u(y)|^p}{|x-y|^{n+sp}} \,dx\,dy  \nonumber \\
&& \qquad  \qquad  \qquad \qquad  \leq  \int_{\Omega}\int_{B_1}\!\frac{|u(x)-u(z+x)|^p}{|z|^{n+sp}} \,dz\,dx \nonumber \\
\nonumber \\
&& \qquad \qquad  \qquad \qquad  =  \int_{\Omega}\int_{B_1}\!\frac{|u(x)-u(z+x)|^p}{|z|^p}\frac{1}{|z|^{n+(s-1)p}} \,dz\,dx\,\nonumber \\
\nonumber \\
&& \qquad  \qquad  \qquad \qquad \leq  \int_{\Omega}\int_{B_1}\left( \int_{0}^1 \frac{|\nabla u (x+tz)|}{|z|^{\frac{n}{p}+s-1 }}\, dt\right)^{p}\! dz\,dx \nonumber \\
\nonumber \\
&& \qquad \qquad  \qquad \qquad   \leq \int_{\R^n} \int_{B_1} \int_{0}^1 \frac{|\nabla \tilde{u} (x+tz)|^p}{|z|^{n+p(s-1) }}\, dt \,dz\,dx\,\nonumber \\
\nonumber \\
&& \qquad \qquad  \qquad \qquad   \leq   \int_{B_1} \int_{0}^1 \frac{\|\nabla \tilde{u}\|^p_{L^p(\R^n)}}{|z|^{n+p(s-1) }}\, dt \,dz \nonumber \\
 \nonumber \\
&& \qquad \qquad  \qquad \qquad   \leq  C_1(n,s,p)\|\nabla \tilde{u} \|^p_{L^p(\R^n)} \nonumber \\
\nonumber\\
&& \qquad  \qquad  \qquad \qquad 
\leq   {C}_2(n,s,p) \|u\|^p_{W^{1,p}(\Omega)}\,.
\end{eqnarray}
\vspace{1mm}

Also, by \eqref{enrico1}, 
\begin{eqnarray}\label{enrico27}
\int_{\Omega}\int_{\Omega\,\cap\,\{|x-y|\geq 1\}}\!\frac{|u(x)-u(y)|^p}{|x-y|^{n+sp}} \,dx\,dy 
&\leq& \!  C(n,s,p) \|u\|^p_{L^p(\Omega)}\,.
\end{eqnarray}
Therefore, from \eqref{enrico26} and \eqref{enrico27} we get estimate \eqref{enrico33}.
\end{proof}

\vspace{2mm}

We remark that the Lipschitz assumption in Proposition \ref{enrico23} cannot be completely dropped (see Example~\ref{exa_lip} in Section~\ref{sec_example}); we also refer to the forthcoming Section~\ref{sec_estensione}, in which we discuss the extension problem in~$W^{s,p}$.

\vspace{3mm}

Let us come back to the definition of the space~$W^{s,p}(\Omega)$. Before going ahead, it is worth explaining why the definition 
in~\eqref{def1} cannot be plainly extended to the case $s \geq 1$. Suppose that $\Omega$ is a connected open set in $\R^n$, then any measurable function $u:\Omega\rightarrow \R$ such that
$$
\int_{\Omega}\, \int_{\Omega} \frac{|u(x)-u(u)|^p}{|x-y|^{n+sp}}\, dx	\, dy \, < \, +\infty
$$
is actually constant (see ~\cite[Proposition 2]{Bre02}).
This fact is a matter of scaling and it is strictly related to the following result that holds for any $u$ in $W^{1,p}(\Omega)$:
\begin{equation}\label{disBre}
\lim_{s\rightarrow 1^{-}} (1-s)\! \int_{\Omega}\int_{\Omega} \frac{|u(x)-u(y)|^p}{|x-y|^{n+sp}}\,dx\,dy\,=\, C_1\! \int_{\Omega} |\nabla u|^p\,dx 
\end{equation}
for a suitable positive constant $C_1$ depending only on $n$ and $p$ (see~\cite{BBM01}).

In the same spirit, in~\cite{MS02},  Maz'ja and Shaposhnikova proved that, for a function~$u \in \bigcup_{0<s<1} W^{s,p}(\R^n)$, it yields
\begin{equation}\label{disMaz}
\lim_{s\rightarrow 0^{+}} s \int_{\R^n}\int_{\R^n} \frac{|u(x)-u(y)|^p}{|x-y|^{n+sp}}\,dx\,dy\,=\, C_2 \int_{\R^n} |u|^p\,dx,
\end{equation}
for a suitable positive constant~$C_2$ depending\footnote{For the sake of simplicity, in the definition of the fractional Sobolev spaces and those of the corresponding norms in~\eqref{def1} and ~\eqref{def2} we avoided any normalization constant. In view of~\eqref{disBre} and \eqref{disMaz}, it is worthing notice that, in order to recover the classical $W^{1,p}$ and $L^p$ spaces, one may consider to add a factor $C(n,p,s)\approx s(1-s)$ in front of the double integral in~\eqref{def2}.
} only on~$n$ and~$p$.

\vspace{3mm}

When $s>1$ and it is not an integer we write $s=m+\sigma$, where $m$ is an integer and $\sigma\in (0,1)$. In this case the space $W^{s,p}(\Omega)$ consists of those equivalence classes of functions $u\in W^{m,p}(\Omega)$ whose distributional derivatives $D^{\alpha}u$, with $|\alpha|=m$, belong to $W^{\sigma,p}(\Omega)$, namely
\begin{equation}\label{def3}
W^{s,p}(\Omega):= \Big\{ u\in W^{m,p}(\Omega)\; :\; D^{\alpha}u \in W^{\sigma,p}(\Omega) \ \text{for any}\ \alpha\ \text{s.t.} \ |\alpha|=m  \Big\}
\end{equation}
and this is a Banach space with respect to the norm
\begin{equation}\label{def4}
\|u\|_{W^{s,p}(\Omega)}:= \left( \|u\|^p_{W^{m,p}(\Omega)} + \sum_{|\alpha|=m} \|D^{\alpha}u\|^p_{W^{\sigma,p}(\Omega)}  \right)^{\!\frac{1}{p}}\!.
\end{equation}
Clearly, if $s=m$ is an integer, the space $W^{s,p}(\Omega)$ coincides with the Sobolev space $W^{m,p}(\Omega)$.
\vspace{2mm}

\begin{corollary}\label{enrico7}
Let $p\in[1,+\infty)$ and $s,s'>1$. Let $\Omega$ be an open set in $\R^n$ of class $C^{0,1}$. Then, if $s'\geq s$, we have
$$
W^{s',p}(\Omega)\subseteq W^{s,p}(\Omega).
$$
\end{corollary}
\begin{proof}
We write $s=k+\sigma$ and $s'=k'+\sigma'$, with $k,k'$ integers and $\sigma,\sigma'\in(0,1)$. In the case $k'=k$, we can use Proposition~\ref{enrico} in order to conclude that $W^{s',p}(\Omega)$ is continuously embedded in $W^{s,p}(\Omega)$. On the other hand, if $k'\geq k+1$, using Proposition~\ref{enrico} and Proposition~\ref{enrico23} we have the following chain
$$ W^{k'+\sigma',p}(\Omega) \subseteq W^{k',p}(\Omega) \subseteq W^{k+1,p}(\Omega) \subseteq W^{k+\sigma,p}(\Omega) \,.
$$
The proof is complete.
\end{proof}
\vspace{2mm}

As in the classic case with $s$ being an integer, any function in the fractional Sobolev space $W^{s,p}(\R^n)$ can be approximated by a sequence of smooth functions with compact support.
\begin{thm}\label{density}
For any $s> 0$,  the space $C_0^{\infty}(\R^n)$ of  smooth functions with compact support is dense in $W^{s,p}(\R^n)$.
\end{thm}
\noindent A proof can be found in \cite[Theorem 7.38]{Ada75}.

\vspace{2mm}
Let $W^{s,p}_0(\Omega)$ denote the closure of $C_0^{\infty}(\Omega)$ in the norm $\|\cdot\|_{W^{s,p}(\Omega)}$ defined in~\eqref{def4}. Note that, in view of Theorem \ref{density}, we have
\begin{equation}\label{dens0}
W_0^{s,p}(\R^n)= W^{s,p}(\R^n)\,,
\end{equation}
but in general, for $\Omega \subset \R^n$, $W^{s,p}(\Omega)\neq W^{s,p}_0(\Omega)$, i.e. 
$C_0^{\infty}(\Omega)$ 
is not dense in $W^{s,p}(\Omega)$.
Furthermore, it is clear that the same inclusions stated in Proposition~\ref{enrico}, Proposition~\ref{enrico23} and Corollary \ref{enrico7} hold for the spaces $W^{s,p}_0(\Omega)$.
\begin{rem}\label{distr}
For $s<0$ and $p \in (1,\infty)$, we can define $W^{s,p}(\Omega)$ as the dual space of $W_0^{-s,q}(\Omega)$ where $1/p+1/q=1$.
Notice that, in this case, the space $W^{s,p}(\Omega)$ is actually a space of distributions on $\Omega$, since it is the dual of a space having $C_0^{\infty}(\Omega)$ as density subset.
\end{rem}

\vspace{1mm}

Finally, it is worth noticing that the fractional Sobolev spaces play an important role in the trace theory. Precisely, for any $p\in (1,+\infty)$, assume that the open set $\Omega\subseteq\R^n$ is sufficiently smooth, then the space of traces $Tu$ on $\partial \Omega$ of $u$ in $W^{1,p}(\Omega)$ is characterized by
$
\dys \|Tu\|_{W^{1-\frac{1}{p},p}(\partial\Omega)}  \! <  +\infty 
$ 
(see \cite{Gag57}). 
Moreover, the trace operator $T$ is surjective from $W^{1,p}(\Omega)$ onto $W^{1-\frac{1}{p},p}(\partial\Omega)$. In the quadratic case $p=2$, the situation simplifies considerably, as we will see in the next section and a proof of the above trace embedding can be find in the forthcoming Proposition~\ref{pro_tracce}.

\vspace{2mm}

\section{The space $H^s$ and the fractional Laplacian operator}\label{sec_laplacian}
In this section, we focus on the case $p=2$. This is quite an important case since the fractional Sobolev spaces $W^{s,2}(\R^n)$ and $W_0^{s,2}(\R^n)$ turn out to be Hilbert spaces. They are usually denoted by $H^s(\R^n)$ and $H_0^s(\R^n)$, respectively. Moreover,
they are strictly related to the fractional Laplacian operator $(-\Delta)^{s}$ (see~Proposition~\ref{prop_frac}),
where, for any $u\in \SF$ and $s\in(0,1)$, $(-\Delta)^s$ is defined as
\begin{eqnarray}\label{def_laplacian}
(-\Delta)^s u(x) & = & C(n,s)\,P.V.\!\!\int_{\R^n} \frac{u(x)-u(y)}{|x-y|^{n+2s}}\, dy  \\
& = &  C(n,s) \lim_{\eps\to 0^+}\int_{\CC B_\eps(x)} \frac{u(x)-u(y)}{|x-y|^{n+2s}}\, dy. \nonumber 
\end{eqnarray}
Here $P.V.$ is a commonly used abbreviation for ``in the principal value sense'' (as defined by the latter equation) and $C(n,s)$ is a dimensional constant that depends on $n$ and $s$, precisely given by
\begin{equation}\label{def_c}
C(n,s)= \left(\int_{\R^n} \frac{1-\cos(\zeta_1)}{|\zeta|^{n+2s}}\,d\zeta \right)^{\!-1}.
\end{equation}
The choice of this constant is motived by Proposition \ref{pro_symbol}.

\vspace{1mm}

\begin{rem}
Due to the singularity of the kernel, the right hand-side of \eqref{def_laplacian} is not well defined in general. In the case $s\in(0,1/2)$ the integral in \eqref{def_laplacian} is not really singular near $x$. Indeed, for any $u\in \SF$, we have
\begin{eqnarray*}
&& \int_{\R^n} \frac{|u(x)-u(y)|}{|x-y|^{n+2s}}\, dy \\
&& \qquad  \qquad \qquad \leq C \int_{B_R} \frac{|x-y|}{\,|x-y|^{n+2s}}\, dy + \|u\|_{L^{\infty}(\R^n)} \int_{\CC B_R} \frac{1}{\,|x-y|^{n+2s}}\, dy \\
&&\qquad  \qquad \qquad = C \left(\int_{B_R} \frac{1}{|x-y|^{n+2s-1}}\, dy  +  \int_{\CC B_R} \frac{1}{\,|x-y|^{n+2s}}\, dy\right) \\
&&\qquad  \qquad \qquad =C \left( \int_{0}^{R}\! \frac{1}{|\rho|^{2s}} \, d\rho \,  +\int_{R}^{+\infty}\! \frac{1}{|\rho|^{2s+1}} \, d\rho \right)\, < \, +\infty
\end{eqnarray*}
where $C$ is a positive constant depending only on the dimension and on the $L^{\infty}$~norm of $u$.
\end{rem}

\vspace{2mm}

Now, we show that one may write the singular integral in~\eqref{def_laplacian} as a weighted second order differential quotient.
\begin{lemma}\label{lem_2nd}
Let $s\in (0,1)$ and let $(-\Delta)^s$ be the fractional Laplacian operator defined by~\eqref{def_laplacian}. Then, for any $u\in \SF$,
\begin{equation}\label{def_2nd}
(-\Delta)^s u(x) = -\frac{1}{2}C(n,s)\int_{\R^n}\frac{u(x+y)+u(x-y)-2u(x)}{|y|^{n+2s}}\, dy, \ \ \ \forall x\in \R^n.
\end{equation}
\end{lemma}
\begin{proof}
The equivalence of the definitions in~\eqref{def_laplacian} and~\eqref{def_2nd} immediately follows by the standard changing variable formula.

\vspace{2mm}

Indeed, by choosing $z=y-x$, we have
\begin{eqnarray}\label{eq_z}
\ds u(x) \! & = & \! - \,C(n,s)\,P.V. \!\int_{\R^n} \frac{u(y)-u(x)}{|x-y|^{n+2s}}\, dy \nonumber \\ \nonumber \\
& = & \! - \,C(n,s)\,P.V.\!\int_{\R^n}\frac{u(x+z)-u(x)}{|z|^{n+2s}}\, dz.
\end{eqnarray}
\vspace{1mm}

Moreover, by substituting $\tilde{z}=-z$ in last term of the above equality, we have
\begin{equation}\label{eq_ztilde}
P.V.\!\int_{\R^n}\!\frac{u(x+z)-u(x)}{|z|^{n+2s}}\, dz \, = \,P.V.\! \int_{\R^n}\!\frac{u(x-\tilde{z})-u(x)}{|\tilde{z}|^{n+2s}}\, d\tilde{z}.
\end{equation}
and so after relabeling $\tilde{z}$ as $z$
\begin{eqnarray}\label{eq_ztilde1}
&& 2 P.V. \! \int_{\R^n}\!\frac{u(x+z)-u(x)}{|z|^{n+2s}}\, dz \nonumber \\
&& \qquad  \qquad \quad = P.V. \! \int_{\R^n}\!\frac{u(x+z)-u(x)}{|z|^{n+2s}}\, dz +P.V.\! \int_{\R^n}\!\frac{u(x-z)-u(x)}{|z|^{n+2s}}\,dz\, \nonumber \\
&& \qquad  \qquad \quad  = P.V.\! \int_{\R^n}\!\frac{u(x+z)+u(x-z)-2u(x)}{|z|^{n+2s}}\, dz.
\end{eqnarray}
Therefore, if we rename $z$ as $y$ in~\eqref{eq_z} and~\eqref{eq_ztilde1}, we can write the fractional Laplacian operator in~\eqref{def_laplacian} as
\begin{equation*}
(-\Delta)^s u(x) \, = \, -\frac{1}{2}\,C(n,s)\,P.V.\!\int_{\R^n}\frac{u(x+y)+u(x-y)-2u(x)}{|y|^{n+2s}}\, dy.
\end{equation*}
The above representation is useful to remove the singularity of the integral at the origin. Indeed, for any smooth function $u$, a second order Taylor expansion yields
$$
\frac{u(x+y)+u(x-y)-2u(x)}{|y|^{n+2s}} \leq \frac{\|D^2u\|_{L^\infty}}{|y|^{n+2s-2}},
$$
which is integrable near $0$ (for any fixed $s\in (0,1)$). Therefore, since $u\in\SF$, one can get rid of the $P.V.$ and write~\eqref{def_2nd}.
\end{proof}

\vspace{3mm}

\subsection{An approach via the Fourier transform}\label{sec_fourier}

Now, we take into account an alternative definition of the space $H^s(\R^n)=W^{s,2}(\R^n)$ 
via the Fourier transform.
Precisely, we may define 
\begin{equation}\label{def_viafourier}
\hat{H}^{s}(\R^n)=\left\{\,u\in L^2(\R^n)\;\, :\;\,\int_{\R^n} (1+|\xi|^{2s}) |\FF u(\xi)|^2\,d\xi < +\infty \, \right\}
\end{equation}
and we observe that the above definition, unlike the ones via the Gagliardo norm in~\eqref{def2}, is valid also for any real $s\geq 1$. 
\vspace{1mm}

We may also use an analogous definition for the case $s<0$ by setting
\begin{equation*}
\hat{H}^{s}(\R^n)=\left\{\,u\in \SF'(\R^n)\;\, :\;\,\int_{\R^n} (1+|\xi|^{2})^{s} |\FF u(\xi)|^2\,d\xi < +\infty \, \right\},
\end{equation*}
although in this case the space $\hat{H}^s(\R^n)$ is not a subset of $L^2(\R^n)$ and, in order to use the Fourier transform, one has to start from an element of $\SF'(\R^n)$, (see also Remark \ref{distr}).

\vspace{1mm}
The equivalence of the space $\hat{H}^s(\R^n)$ defined  in~\eqref{def_viafourier} with the one defined in the previous section via the Gagliardo norm (see~\eqref{def1}) is stated and proven in the forthcoming Proposition~\ref{pro_equiv}.

\vspace{2mm}

First, we will prove that the fractional Laplacian $\ds$ can be viewed as a pseudo-differential operator of symbol $|\xi|^{2s}$. The proof is standard and it can be found in many papers (see, for instance, \cite[Chapter 16]{Tar07}). We will follow the one in~\cite{Val09} (see Section 3), in which is shown how singular integrals naturally arise as a continuous limit of discrete long jump random walks.

\begin{prop}\label{pro_symbol}
Let $s\in (0,1)$ and let $(-\Delta)^s:\SF\to L^2(\R^n)$ be the fractional Laplacian operator defined by~\eqref{def_laplacian}. Then, for any $u\in \SF$,
\begin{equation}\label{def_confourier}
(-\Delta)^s u \, = \, \FF^{-1}(|\xi|^{2s}(\FF u)) \ \ \ \forall \xi\in \R^n.
\end{equation}
\end{prop}
\begin{proof}
In view of Lemma~\ref{lem_2nd}, we may use the definition via the weighted second order differential quotient in ~\eqref{def_2nd}.
We denote by $\LL u$ the integral in~\eqref{def_2nd}, that is
$$
\LL u (x) = -\frac{1}{2}\, C(n,s)\,\int_{\R^n}\frac{u(x+y)+u(x-y)-2u(x)}{|y|^{n+2s}}\, dy,
$$
with $C(n,s)$ as in \eqref{def_c}.\\
 $\LL$ is a linear operator and we are looking for its ``symbol'' (or ``multiplier''), that is a function $\mathcal{S}:\R^n\to\R$ such that
\begin{equation}\label{eq_anti}
\LL u = \FF^{-1}(\mathcal{S}(\FF u)).
\end{equation}

We want to prove that
\begin{equation}\label{eq_symbolbut}
\mathcal{S}(\xi) \, = \, |\xi|^{2s},
\end{equation}
where we denoted by $\xi$ the frequency variable.
\vspace{2mm}
To this scope,
we point out that
\begin{eqnarray*}
&& \!\!\!\!\!\!\!\!\! \frac{|u(x+y)+u(x-y)-2u(x)|}{|y|^{n+2s}}\\
&& \quad \le \, 4\Big( \chi_{B_1}(y) |y|^{2-n-2s} \sup_{B_1(x)} |D^2u|
+\chi_{\R^n\setminus B_1}(y) |y|^{-n-2s} \sup_{\R^n} |u|\Big)\\
&& \quad \le\, C\Big( \chi_{B_1}(y) |y|^{2-n-2s} (1+|x|^{n+1})^{-1}+
\chi_{\R^n\setminus B_1}(y) |y|^{-n-2s} \Big)
\,\in \, L^1 (\R^{2n}).
\end{eqnarray*}
Consequently, by the Fubini-Tonelli's Theorem, we can exchange
the integral in~$y$ with the Fourier transform in~$x$.
Thus, we apply the Fourier transform in the variable $x$ in~\eqref{eq_anti} and we obtain
\begin{eqnarray}\label{eq_12}
\mathcal{S}(\xi)(\FF u)(\xi) \! & = & \! \FF(\LL u) 
\nonumber \\[1ex]
& = & \! -\frac{1}{2}\, C(n,s)\int_{\R^n} \frac{\FF\left(u(x+y)+u(x-y)-2u(x)\right)}{|y|^{n+2s}}\, dy \nonumber \\
\nonumber \\
& = & \! -\frac{1}{2}\, C(n,s) \int_{\R^n} \frac{e^{i\xi\cdot y}+ e^{-i\xi\cdot y} -2}{|y|^{n+2s}}\, dy (\FF u)(\xi) \nonumber \\
\nonumber \\
& = & \! C(n,s)\int_{\R^n} \frac{1-\cos(\xi\cdot y)}{|y|^{n+2s}}\,  dy (\FF u)(\xi).
\end{eqnarray}
Hence, in order to obtain~\eqref{eq_symbolbut}, it suffices to show that
\begin{equation}\label{eq_15}
\int_{\R^n}\frac{1-\cos(\xi\cdot y)}{|y|^{n+2s}}\, dy \, = \, C(n,s)^{-1}|\xi|^{2s} .
\end{equation}
To check this, first we observe that, if $\zeta=(\zeta_1, ..., \zeta_n) \in \R^n$, we have
$$
\frac{1-\cos{\zeta_1}}{|\zeta|^{n+2s}} \, \le \, \frac{|\zeta_1|^2}{|\zeta|^{n+2s}} \, \le \, \frac{1}{|\zeta|^{n-2+2s}}
$$
near $\zeta=0$. Thus,
\begin{equation}\label{eq_16}
\int_{\R^n}\frac{1-\cos{\zeta_1}}{|\zeta|^{n+2s}}\, d\zeta \ \, \text{is finite and positive.}
\end{equation}
Now, we consider the function $\mathcal{I}:\R^n\to \R$ defined as follows
$$
\mathcal{I}(\xi)=\int_{\R^n}\frac{1-\cos{(\xi\cdot y)}}{|y|^{n+2s}}dy.
$$
We have that $\mathcal{I}$ is rotationally invariant, that is
\begin{equation}\label{eq_17}
\mathcal{I}(\xi) = \mathcal{I}(|\xi|e_1),
\end{equation}
where $e_1$ denotes the first direction vector in $\R^n$.
Indeed, when $n=1$, then we can deduce~\eqref{eq_17} by the fact that $\mathcal{I}(-\xi)=\mathcal{I}(\xi)$. When $n\geq 2$, we consider a rotation $R$ for which $\dys R(|\xi|e_1)=\xi$ and we denote by $R^T$ its transpose. Then, by substituting $\tilde{y}=R^Ty$, we obtain
\begin{eqnarray*}
\mathcal{I}(\xi)\! & = & \! \int_{\R^n}\frac{1-\cos{\big((R(|\xi|e_1))\cdot y\big)}}{|y|^{n+2s}}\, dy \\
& = & \! \int_{\R^n}\frac{1-\cos{\big((|\xi|e_1)\cdot (R^Ty)\big)}}{|y|^{n+2s}}\, dy \\
& = & \! \int_{\R^n}\frac{1-\cos{\big((|\xi|e_1)\cdot \tilde{y}\big)}}{|\tilde{y}|^{n+2s}}\, d\tilde{y} \ = \ \mathcal{I}(|\xi|e_1),
\end{eqnarray*}
which proves~\eqref{eq_17}.
\vspace{1mm}

As a consequence of~\eqref{eq_16} and~\eqref{eq_17}, the substitution $\zeta=|\xi|y$ gives that
\begin{eqnarray*}
\mathcal{I}(\xi) \! & = & \! \mathcal{I}(|\xi|e_1) \\
& = & \! \int_{\R^n}\frac{1-\cos{(|\xi|y_1)}}{|y|^{n+2s}}\, dy \\
& = & \! \frac{1}{|\xi|^n}\int_{\R^n}\frac{1-\cos{\zeta_1}}{\big|\zeta/|\xi|\big|^{n+2s}}\, d\zeta \ = \ C(n,s)^{-1}|\xi|^{2s}.
\end{eqnarray*}
where we recall that $C(n,s)^{-1}$ is equal to $\dys \int_{\R^n} \frac{1-\cos(\zeta_1)}{|\zeta|^{n+2s}}\,d\zeta\,$ by \eqref{def_c}.
Hence,~we deduce \eqref{eq_15} and then the proof is complete.
\end{proof}
\vspace{1mm}

\begin{prop}\label{pro_equiv}
Let $s\in (0,1)$. Then the fractional Sobolev space $H^s(\R^n)$ defined in Section \ref{definizioni} 
coincides with $\hat{H}^s(\R^n)$ defined in~\eqref{def_viafourier}. In particular, for any $u\in H^s(\R^n)$
\begin{equation*}
[u]^2_{H^s(\R^n)} \, = \,  2C(n,s)^{-1}\!\int_{\R^n} |\xi|^{2s} |\FF u(\xi)|^2\, d\xi.
\end{equation*}
where $\dys C(n,s)$ is defined by~\eqref{def_c}.
\end{prop}

\begin{proof}
For every fixed $y\in \R^n$, by changing of variable choosing $z=x-y$, we get
\begin{eqnarray*}
\int_{\R^n}\,\left(\int_{\R^n} \frac{|u(x)-u(y)|^2}{|x-y|^{n+2s}}\,dx\right)\,dy \!&=&\!\int_{\R^n} \int_{\R^n} \frac{|u(z+y)-u(y)|^2}{|z|^{n+2s}} \,dz\,dy \nonumber \\
&=& \!\int_{\R^n} \left(\int_{\R^n} \left|\frac{u(z+y)-u(y)}{|z|^{n/2+s}}\right|^2 \,dy\,\right)dz \nonumber\\
&=& \!\int_{\R^n} \, \left\| \frac{u(z+\cdot)-u(\cdot)}{|z|^{n/2+s}} \right\|^2_{L^2(\R^n)}  dz \nonumber\\
&=& \!\int_{\R^n} \, \left\| \FF \left(\frac{u(z+\cdot)-u(\cdot)}{|z|^{n/2+s}} \right) \right\|^2_{L^2(\R^n)} dz, \nonumber
\end{eqnarray*}
where Plancherel Formula has been used.
\vspace{2mm}

Now, using \eqref{eq_15} we obtain
\begin{eqnarray*}
&& \int_{\R^n}  \left\| \FF \left(\frac{u(z+\cdot)-u(\cdot)}{|z|^{n/2+s}} \right) \right\|^2_{L^2(\R^n)}   \!dz \nonumber \\
\\
&& \qquad \qquad \qquad \qquad = \int_{\R^n} \int_{\R^n} \frac{|e^{i\xi \cdot z}-1|^2}{|z|^{n+2s}}\, |\FF u(\xi)|^2 \, d\xi   \,dz\\
\\
&&  \qquad \qquad \qquad \qquad = \, 2\,\int_{\R^n} \int_{\R^n} \frac{(1-\cos \xi \cdot z)}{|z|^{n+2s}}\, |\FF u(\xi)|^2 \, dz\, d\xi  \\
\\
&&  \qquad \qquad \qquad \qquad = \, 2C(n,s)^{-1} \int_{\R^n} |\xi|^{2s}\, |\FF u(\xi)|^2 \, d\xi.
\end{eqnarray*}

This completes the proof.\end{proof}

\vspace{1mm}

\begin{rem}
The equivalence of the spaces $H^s$ and $\hat{H}^s$ stated in~Proposition~\ref{pro_equiv} relies on Plancherel Formula. 
As well known, unless $p=q=2$, one cannot go forward and backward between an $L^p$ and an $L^q$ via Fourier transform (see, for instance, the sharp inequality in~\cite{Bec75} for the case $1<p<2$ and $q$ equal to the conjugate exponent $p/(p-1)$\,). That is why the general fractional space defined via Fourier transform for $1<p<\infty$ and $s>0$, say  ${H}^{s,p}(\R^n)$, does not coincide with the fractional Sobolev spaces $W^{s,p}(\R^n)$ and will be not discussed here (see, e.g., \cite{Zie89}).
\end{rem}
\vspace{1mm}
Finally, we are able to prove the relation between the fractional Laplacian operator $(-\Delta)^s$ and the fractional Sobolev space $H^s$.

\begin{prop}\label{prop_frac}
Let $s\in (0,1)$ and let $u\in H^s(\R^n)$. Then,
\begin{equation}\label{aaaeleono}
[u]^2_{H^s(\R^n)}\,=\, 2C(n,s)^{-1}\|(-\Delta)^{\frac{s}{2}}u\|^2_{L^2(\R^n)}.
\end{equation}
where $ C(n,s)$ is defined by~\eqref{def_c}.
\end{prop}
\begin{proof} The equality in \eqref{aaaeleono} plainly follows from Proposition~\ref{pro_symbol} and Proposition~\ref{pro_equiv}. Indeed, 
\begin{eqnarray*}
\| (-\Delta)^{\frac{s}{2}} u \|^2_{L^2(\R^n)} \!&=&\! \| \FF (-\Delta)^{\frac{s}{2}} u \|^2_{L^2(\R^n)} \, = \, \| |\xi|^s \FF u \|^2_{L^2(\R^n)} \\
\\
& =& \! \frac{1}{2}C(n,s) [ u ]^2_{{H}^s(\R^n)}\,.\qedhere
\end{eqnarray*}
\end{proof}

\vspace{0.5mm}

\begin{rem}
In the same way as the fractional Laplacian~$(-\Delta)^s$ is related to the space $W^{s,2}$ \big(as its Euler-Lagrange equation or from the formula $\|u\|^2_{W^{s,2}} = \int u (-\Delta)^s u\,dx$\,\big), a more general integral operator can be defined that is related to the space $W^{s,p}$ for any $p$ (see the recent paper~\cite{IN10}).
\end{rem}

\vspace{3mm}

Armed with the definition of $H^s(\R^n)$ via the Fourier transform, we can easily analyze the traces of the Sobolev functions (see the forthcoming~Proposition~\ref{pro_tracce}). We will follow Sections 13, 15 and 16 in~\cite{Tar07}.

\vspace{2mm}

Let $\Omega\subseteq\R^n$ be an open set with continuous boundary $\partial\Omega$. Denote by $T$ the {\it trace operator}, namely the linear operator defined by the uniformly continuous extension of the operator of restriction to $\partial\Omega$ for functions in~$\mathcal{D}(\overline{\Omega})$, that is the space of functions~$C^{\infty}_0(\R^n)$ restricted\footnote{
Notice that we cannot simply take $T$ as the restriction operator to the boundary, since the restriction to a set of measure $0$ (like the set $\partial\Omega$) is not defined for functions which are not smooth enough.
}
to $\overline{\Omega}$.

\vspace{2mm}

Now, for any $x=(x',x_n)\in \R^n$ and for any $u \in \SF(\R^n)$, we denote by $v\in\SF(\R^{n-1})$ the restriction of $u$ on the hyperplane $x_n=0$, that is
\begin{equation}\label{def_traccia}
\dys v(x')=u(x',0) \ \ \ \forall x'\in \R^{n-1}.
\end{equation}
Then, we have
\begin{equation}\label{eq_tar1520}
\dys
\FF v(\xi') \, = \, \int_{\R} \FF u(\xi',\xi_n)\, d\xi_n \ \ \ \forall \xi' \in \R^{n-1},
\end{equation}
where, for the sake of simplicity, we keep the same symbol $\FF$ for both the Fourier transform in $n-1$ and in $n$ variables.
\vspace{2mm}

To check~\eqref{eq_tar1520}, we write
\begin{eqnarray}\label{eq_m1}
\dys
\FF v(\xi') \! & = & \! \frac{1}{(2\pi)^{\frac{n-1}{2}}}\int_{\R^{n-1}} e^{-i\xi'\cdot x'} v(x')\, dx' \nonumber \\
\nonumber \\
& = & \! \frac{1}{(2\pi)^{\frac{n-1}{2}}}\int_{\R^{n-1}} e^{-i\xi'\cdot x'} u(x',0)\, dx'.
\end{eqnarray}
\vspace{1mm}

On the other hand, we have
\begin{eqnarray}\label{eq_m2}
\dys
&& \int_\R \FF u(\xi',\xi_n)\, d\xi_n \! \! \nonumber \\
\nonumber \\
&& \qquad \quad  =\ \int_\R \frac{1}{(2\pi)^{\frac{n}{2}}}\int_{\R^n}\!e^{-i\,(\xi',\xi_n)\cdot(x', x_n)} u(x',x_n)\, dx'\, dx_n\, d\xi_n \nonumber \\
\nonumber \\
&& \qquad \quad = \ \frac{1}{(2\pi)^{\frac{n-1}{2}}}\!\int_{\R^{n-1}}\!e^{-i \xi'\cdot x'}\!\left[ \frac{1}{(2\pi)^{\frac{1}{2}}}\!\int_\R \!\int_\R e^{-i \xi_n\cdot x_n} u(x', x_n)\, d x_n d\xi_n\right] \!dx' \nonumber \\
\nonumber \\
&& \qquad \quad = \ \frac{1}{(2\pi)^{\frac{n-1}{2}}}\!\int_{\R^{n-1}}\!e^{-i \xi'\cdot x'}\!\big[ u(x',0)\big] dx', \nonumber
\end{eqnarray}
where the last equality follows by transforming and anti-transforming $u$ in the last variable, and this coincides with~\eqref{eq_m1}.

\vspace{1mm}

Now, we are in position to characterize the traces of the function in~$H^s(\R^n)$, as stated in the following proposition.
\begin{prop}\label{pro_tracce}{\rm (\cite[Lemma 16.1]{Tar07}).}
Let $s>1/2$, then any function $u \in H^s(\R^n)$ has a trace $v$ on the hyperplane $\big\{x_n=0\big\}$, such that~$v \in H^{s-\frac{1}{2}}(\R^{n-1})$. Also, the trace operator $T$ is surjective from $H^s(\R^n)$ onto $H^{s-\frac{1}{2}}(\R^{n-1})$.
\end{prop}
\begin{proof}
In order to prove the first claim, it suffices to show that there exists an universal constant $C$ such that, for any $u\in \SF(\R^n)$ and any $v$ defined as in~\eqref{def_traccia},
\begin{equation}\label{eq_trasob}
\dys
\|v\|_{H^{s-\frac{1}{2}}(\R^{n-1})} \, \leq \, C \|u\|_{H^s(\R^n)}.
\end{equation}
By taking into account~\eqref{eq_tar1520}, the Cauchy-Schwarz inequality yields
\begin{equation}\label{eq_tar161}
\dys
|\FF v(\xi')|^2 \, \leq \, \left(\int_{\R}(1+|\xi|^2)^s|\FF u(\xi',\xi_n)|^2\, d\xi_n\right)\left(\int_{\R}\frac{d\xi_n}{(1+|\xi|^2)^s}\right).
\end{equation}
Using the changing of variable formula by setting $\dys \xi_n = t\sqrt{1+|\xi'|^2}$, we have
\begin{eqnarray}\label{eq_tar162}
\dys
\int_{\R}\frac{d\xi_n}{(1+|\xi|^2)^s} \! &  = & \! \int_{\R}\frac{\big(1+|\xi'|^2\big)^{1/2}}{\big((1+|\xi'|^2)(1+t^2)\big)^s}\, dt \  = \  \int_{\R} \frac{\big(1+|\xi'|\big)^{\frac{1}{2}-s}}{(1+t^2)^s}\, dt \nonumber \\
\nonumber \\
& = & \! C(s) \big(1+|\xi'|^2\big)^{\frac{1}{2}-s},
\end{eqnarray}
where $\dys C(s):=\int_\R \frac{dt}{(1+t^2)^s} \, < \, +\infty$ since $s>1/2$.
\vspace{1mm}

Combining~\eqref{eq_tar161} with~\eqref{eq_tar162} and integrating in~$\xi'\in\R^{n-1}$, we obtain
\begin{eqnarray*}
&& \dys
\int_{\R^{n-1}}\big(1+|\xi'|^2\big)^{s-\frac{1}{2}}|\FF v(\xi')|^2\, d\xi' \\
&&  \qquad \qquad \qquad \quad \leq \, C(s)\int_{\R^{n-1}}\int_\R \big(1+|\xi|^2\big)^s |\FF u(\xi',\xi_n)|^2\, d\xi_n\, d\xi',
\end{eqnarray*}
that is~\eqref{eq_trasob}.
\vspace{2mm}

Now, we will prove the surjectivity of the trace operator $T$. For this, we show that for any $v\in H^{s-\frac{1}{2}}(\R^{n-1})$ the function $u$ defined by
\begin{equation}\label{eq_tar164}
\dys
\FF u(\xi',\xi_n) \, = \, \FF v(\xi')\, \varphi\!\left( \frac{\xi_n}{\sqrt{1+|\xi'|^2}}\right)\frac{1}{\sqrt{1+|\xi'|^2}},
\end{equation}
with $\varphi\in C^{\infty}_0(\R)$ and $\dys \int_\R \varphi(t)\, dt = 1$, is such that $u\in H^s(\R^n)$ and $Tu=v$. Indeed, we integrate~\eqref{eq_tar164} with respect to $\xi_n\in\R$, we substitute $\xi_n= t \sqrt{1+|\xi'|^2}$ and we obtain
\begin{eqnarray}\label{eq_stellina}
\dys
\int_{\R}\FF u(\xi',\xi_n)\, d\xi_n \! & = & \! \int_{\R}\FF v(\xi')\, \varphi\!\left( \frac{\xi_n}{\sqrt{1+|\xi'|^2}}\right)\frac{1}{\sqrt{1+|\xi'|^2}}\, d\xi_n \nonumber \\
\nonumber \\
& = & \! \int_{\R}\FF v(\xi')\, \varphi(t)\,dt \ = \ \FF v(\xi')
\end{eqnarray}
and this implies $v= Tu$ because of~\eqref{eq_tar1520}.
\vspace{1mm}

The proof of the $H^s$-boundedness of $u$ is straightforward. In fact, from \eqref{eq_tar164}, for any $\xi'\in \R^{n-1}$, we have
\begin{eqnarray}\label{eq_tar165}
\dys
&& \int_\R \big(1+|\xi|^2\big)^s |\FF u (\xi', \xi_n)|^2\, d\xi_n \nonumber \\
&& \qquad \qquad \qquad  =  \ \int_\R \big(1+|\xi|^2\big)^s |\FF v(\xi')|^2  \left|\varphi\!\left(\frac{\xi_n}{\sqrt{1+|\xi'|^2}}\right)\right|^2\!\frac{1}{{1+|\xi'|^2}}\,d\xi_n \nonumber \\
 \nonumber \\
& & \qquad \qquad \qquad = \ C\big(1+|\xi'|^2)^{s-\frac{1}{2}} |\FF v(\xi')|^2,
\end{eqnarray}
where we used again the changing of variable formula with $\dys \xi_n= t\sqrt{1+|\xi'|^2}$ and
the constant~$C$ is given by $\dys \int_\R\big(1+t^2\big)^s |\varphi(t)|^2\, dt$. 
Finally, we obtain that $u\in H^s(\R^n)$ by integrating~\eqref{eq_tar165} in $\xi'\in\R^{n-1}$.
\end{proof}

\vspace{3mm}

\begin{rem}
We conclude this section by recalling that the fractional Laplacian $(-\Delta)^s$, which is a nonlocal operator on functions defined in~$\R^n$, may be reduced to a local, possibly singular or degenerate, operator on functions sitting in the higher dimensional half-space~$\R^{n+1}_+=\R^n\times(0,+\infty)$. We have
$$
(-\Delta)^s u(x) \, = \, - C\lim_{t\to 0}\left(t^{1-2s}\frac{\partial U}{\partial t}(x,t)\right),
$$
where the function $U:\R^{n+1}_+\to\R$ solves $\textrm{div}(t^{1-2s}\nabla U)=0$ in $\R^{n+1}_+$ and $U(x,0)=u(x)$ in~$\R^n$.
\vspace{1mm}

This approach was pointed out by Caffarelli and Silvestre~in~\cite{CS07}; see, in particular, Section 3.2 there, where was also given an equivalent definition of the $H^s(\R^n)$-norm:
$$
\dys
\int_{\R^n} |\xi|^{2s} |\FF u|^2\, d\xi \, = \, C \int_{\R^{n+1}_+} |\nabla U|^2 t^{1-2s}\, dx\, dt.
$$

The cited results turn out to be very fruitful in order to recover
an elliptic PDE approach in a nonlocal framework,
and they have recently been used very often (see, 
e.g.,~\cite{CSS08, SiV09, CC10, CRS10b, PS10}, etc.).

\end{rem}

\vspace{3mm}

\section{Asymptotics of the constant $C(n,s)$}\label{sec_cost}
In this section, we go into detail on the constant factor $C(n,s)$ that appears in the definition of the fractional Laplacian (see~\eqref{def_laplacian}), by analyzing 
 its asymptotic behavior as $s\to 1^{-}$ and $s\to 0^{+}$. This is relevant if one wants to recover the Sobolev norms of the spaces $H^{1}(\R^n)$ and $L^2(\R^n)$ by starting from the one of~$H^s(\R^n)$. 
\vspace{1mm}

We recall that in Section~\ref{sec_laplacian}, the constant~$C(n,s)$ has been defined by
$$
C(n,s) = \left( \int_{\R^n} \frac{1-\cos(\zeta_1)}{|\zeta|^{n+2s}}\,d\zeta \right)^{\!-1}.
$$
Precisely, we are interested in analyzing 
the asymptotic behavior as $s\to 0^+$ and $s\to 1^-$ of a  scaling of the quantity in the right hand-side of the above formula.

\vspace{1mm}

By changing variable $\eta'=\zeta^{'}/|\zeta_1|$, we have
\begin{eqnarray*}
\int_{\R^n} \frac{1-\cos(\zeta_1)}{|\zeta|^{n+2s}}\,d\zeta
 &=& \int_{\R} \int_{\R^{n-1}} \frac{1-\cos(\zeta_1)}{|\zeta_1|^{n+2s}} 
      \frac{1}{(1+ |\zeta'|^2/|\zeta_1|^2)^{\frac{n+2s}{2}}} \,d\zeta' d\zeta_1  \\
 &=& \int_{\R} \int_{\R^{n-1}} \frac{1-\cos(\zeta_1)}{|\zeta_1|^{1+2s}} 
      \frac{1}{(1+ |\eta'|^2)^{\frac{n+2s}{2}}} \,d\eta' d\zeta_1  \\ 
 \vspace{2mm}     
 &=& \frac{A(n,s)\, B(s)}{s(1-s)}      
\end{eqnarray*}
where 
\begin{equation}\label{a}
A(n,s)= \int_{\R^{n-1}} \frac{1}{(1+ |\eta'|^2)^{\frac{n+2s}{2}}} \,d\eta'
\end{equation}
 and\footnote{Of course, when $n=1$ (\ref{a}) reduces to $A(n,s)=1$, so we will just consider the case $n>1$.}
\begin{equation}\label{b} 
B(s)= s(1-s)\int_{\R} \frac{1-\cos t}{|t|^{1+2s}} \,dt.
\end{equation}
\vspace{1mm}
\begin{prop}\label{asy_est}
For any $n> 1$, let $A$ and $B$ be defined by~\eqref{a} and~\eqref{b} respectively. The following statements hold:
\begin{itemize}
\item[{\rm $(i)$}] $\dys \lim_{s\to 1^{-}} A(n,s)= \omega_{n-2} \int_{0}^{+\infty} \!\frac{\rho^{n-2}}{(1+ \rho^2)^{\frac{n}{2}+1}} \,d\rho < + \infty$;
\item[{\rm $(ii)$}] $\dys \lim_{s\to 0^+} A(n,s)= \omega_{n-2} \int_{0}^{+\infty}\! \frac{\rho^{n-2}}{(1+ \rho^2)^{\frac{n}{2}}}\,d\rho < +\infty$;
\item[{\rm $(iii)$}] $\dys \lim_{s\to 1^{-}} B(s)=\frac{1}{2}$;
\item[{\rm $(iv)$}] $\dys \lim_{s\to 0^{+}} B(s)=1,$
\end{itemize}
where $\omega_{n-2}$ denotes $(n-2)$-dimensional measure of the unit sphere $S^{n-2}$. \\
As a consequence,
\begin{equation}\label{cost1}
\lim_{s\to 1^{-}} \frac{C(n,s)}{s(1-s)}= \left( \frac{\omega_{n-2}}{2} \int_{0}^{+\infty} \!\frac{\rho^{n-2}}{(1+ \rho^2)^{\frac{n}{2}+1}} \,d\rho \right)^{\!-1}
\end{equation}
and
\begin{equation}\label{cost2}
 \lim_{s\to 0^{+}} \frac{C(n,s)}{s(1-s)}= \left(\omega_{n-2} \int_{0}^{+\infty}\! \frac{\rho^{n-2}}{(1+ \rho^2)^{\frac{n}{2}}} \,d\rho \right)^{\!-1}.
\end{equation}
\end{prop}
\begin{proof}
First, by polar coordinates, for any $s\in (0,1)$, we get
\begin{eqnarray*}
\int_{\R^{n-1}} \frac{1}{(1+ |\eta'|^2)^{\frac{n+2s}{2}}} \,d\eta' =  \omega_{n-2} \int_{0}^{+\infty}\! \frac{\rho^{n-2}}{(1+ \rho^2)^{\frac{n+2s}{2}}}\,d\rho.
\end{eqnarray*}
Now, observe that for any $s\in (0,1)$ and any $\rho\geq0$, we have
$$
\frac{\rho^{n-2}}{(1+ \rho^2)^{\frac{n+2s}{2}}} \, \leq\, \frac{\rho^{n-2}}{(1+ \rho^2)^{\frac{n}{2}}}
$$
and the function in the right hand-side of the above inequality belongs to~$L^1((0,+\infty))$ for any~$n>1$.
\vspace{1mm}

Then, the Dominated Convergence Theorem yields
\begin{eqnarray*}
\lim_{s\to 1^-} A(n,s)& = & \omega_{n-2} \int_{0}^{+\infty}\! \frac{\rho^{n-2}}{(1+ \rho^2)^{\frac{n}{2}+1}} \,d\rho
\end{eqnarray*}
and
\begin{eqnarray*}
\lim_{s\to 0^+} A(n,s)&= & \omega_{n-2} \int_{0}^{+\infty}\! \frac{\rho^{n-2}}{(1+ \rho^2)^{\frac{n}{2}}}\,d\rho.
\end{eqnarray*}
\vspace{2mm}
This proves $(i)$ and $(ii)$. \\
Now, we want to prove $(iii)$. First, we split the integral in~\eqref{b} as follows
$$ 
\int_{\R} \frac{1-\cos t}{|t|^{1+2s}} \,dt= \int_{|t|< 1} \frac{1-\cos t}{|t|^{1+2s}} \,dt + 
 \int_{|t|\geq 1} \frac{1-\cos t}{|t|^{1+2s}} \,dt.
$$
Also, we have that
$$
0 \leq \int_{|t|\geq 1} \frac{1-\cos t}{|t|^{1+2s}} \,dt  \leq 4 \int_{1}^{+\infty} \frac{1}{t^{1+2s}}\,dt=\frac{2}{s}
$$
and
$$
 \int_{|t|<1} \frac{1-\cos t}{|t|^{1+2s}} \,dt - \int_{|t|<1} \frac{t^2}{2|t|^{1+2s}} \,dt \,
 \leq \, C \int_{|t|<1} \frac{|t|^3}{|t|^{1+2s}} \,dt \, = \, \frac{2C}{3-2s},
$$
for some suitable positive constant $C$.

\vspace{1mm}

From the above estimates it follows that 
$$\lim_{s\to 1^-} s(1-s) \int_{|t|\geq 1} \frac{1-\cos t}{|t|^{1+2s}} \,dt = 0
$$
and
$$
\lim_{s\to 1^-}  s(1-s) \int_{|t|<1} \frac{1-\cos t}{|t|^{1+2s}} \,dt  = \lim_{s\to 1^-}  s(1-s) \int_{|t|<1} \frac{t^2}{2|t|^{1+2s}} \,dt.
$$

Hence, we get
$$
\lim_{s\to 1^-} B(s)= \lim_{s\to 1^-} s(1-s) \left( \int_{0}^1 t^{1-2s} \,dt \right)=\lim_{s\to 1^-} \frac{s(1-s)}{2(1-s)}=\frac{1}{2}.
$$
\vspace{2mm}

Similarly, we can prove $(iv)$. For this we notice that 
$$
0 \leq  \int_{|t|<1} \frac{1-\cos t}{|t|^{1+2s}} \,dt  \, \leq C \! \int_{0}^{1} t^{1-2s} \,dt
$$
which yields
$$
\lim_{s\to 0^+} s(1-s) \int_{|t|<1} \frac{1-\cos t}{|t|^{1+2s}} \,dt \, =\, 0.
$$
\vspace{1mm}

Now, we observe that for any~$k\in\N$, $k\ge1$, we have
\begin{eqnarray*}
&& \left| \int_{2k\pi}^{2(k+1)\pi} \frac{\cos t}{t^{1+2s}}\,dt\right|\,=\,
\left| \int_{2k\pi}^{2k\pi+\pi} \frac{\cos t}{t^{1+2s}}\,dt
+
\int_{2k\pi}^{2k\pi+\pi} \frac{\cos (\tau+\pi)}{(\tau+\pi)^{1+2s}}\,d\tau
\right|\\
&&\qquad= \, 
\left| \int_{2k\pi}^{2k\pi+\pi} \cos t\left( \frac{1}{t^{1+2s}}
-\frac{1}{(t+\pi)^{1+2s}}\right)dt\right|\\
&&\qquad\le \,
\int_{2k\pi}^{2k\pi+\pi} \left| \frac{1}{t^{1+2s}}
-\frac{1}{(t+\pi)^{1+2s}}\right|\,dt
\\ &&\qquad=\,
\int_{2k\pi}^{2k\pi+\pi} 
\frac{(t+\pi)^{1+2s}-t^{1+2s}}{t^{1+2s}(t+\pi)^{1+2s}}\,dt
\\ &&\qquad= \,
\int_{2k\pi}^{2k\pi+\pi} 
\frac{1}{t^{1+2s}(t+\pi)^{1+2s}}\left( \int_0^\pi (1+2s) (t+\vartheta)^{2s}\,d\vartheta\right)dt      
\\ &&\qquad\le \,
\int_{2k\pi}^{2k\pi+\pi}
\frac{3\pi(t+\pi)^{2s}}{t^{1+2s}(t+\pi)^{1+2s}}\,dt\\
&& \qquad \le \, \int_{2k\pi}^{2k\pi+\pi}
\frac{3\pi}{t (t+\pi)}\,dt
\\ &&\qquad    
\le\, \int_{2k\pi}^{2k\pi+\pi}
\frac{3\pi}{t^2}\,dt
\, \le \,\frac{C}{k^2}.
\end{eqnarray*}
As a consequence,
\begin{eqnarray*}
\dys \left| \int_{1}^{+\infty} \frac{\cos t}{t^{1+2s}}\,dt\right|
\! & \le& \!
\int_{1}^{2\pi} \frac{1}{t}\,dt+\left|\sum_{k=1}^{+\infty}
\int_{2k\pi}^{2(k+1)\pi} \frac{\cos t}{t^{1+2s}}\,dt\right|\\
\\
& \le & \!\log(2\pi)+\sum_{k=1}^{+\infty} \frac{C}{k^2}
\, \le \, C,
\end{eqnarray*}
up to relabeling the constant~$C>0$.
\vspace{1mm}

It follows that
\begin{eqnarray*}
\left|\int_{|t|\geq 1} \frac{1-\cos t}{|t|^{1+2s}}\, dt -  \int_{|t|\geq 1} \frac{1}{|t|^{1+2s}}\, dt \right|
\! & =  & \! \left| \int_{|t|\geq 1} \frac{\cos t}{|t|^{1+2s}}\, dt \right| \\
\\
\! & = & \! 2 \left| \int_{1}^{+\infty} \frac{\cos t}{t^{1+2s}}\, dt \right| \, \leq \, C
\end{eqnarray*}
and then
$$
\dys \lim_{s\to 0^+} s(1-s)\int_{|t|\geq 1} \frac{1-\cos t}{|t|^{1+2s}} \,dt
\,= \, \lim_{s\to 0^+} s(1-s)\int_{|t|\geq 1} \frac{1}{|t|^{1+2s}}\, dt.
$$
\vspace{1mm}

Hence, we can conclude that
\begin{eqnarray*}
\lim_{s\to 0^+} B(s) \! &=& \! \lim_{s\to 0^+} s(1-s)\int_{|t|\geq 1} \frac{1}{|t|^{1+2s}}\, dt  \\
&=& \! \lim_{s\to 0^+} 2s(1-s)\int_{1}^{+\infty} t^{-1-2s}\, dt \\
&=& \! \lim_{s\to 0^+} \frac{2s(1-s)}{2s}=1.
\end{eqnarray*}\vspace{1mm}

Finally, (\ref{cost1}) and (\ref{cost2}) easily follow combining the previous estimates and recalling that
$$C(n,s)=\frac{s(1-s)}{A(n,s) B(s)}.$$

The proof is complete.\end{proof}

\begin{corollary}\label{asy_est1}
For any $n> 1$, let $C(n,s)$ be defined by~\eqref{def_c}. The following statements hold:
\begin{itemize}
\item[{\rm $(i)$}] $\dys \lim_{s\to 1^{-}} \frac{C(n,s)}{s(1-s)}= \frac{4n}{\omega_{n-1}}$;
\item[{\rm $(ii)$}] $\dys \lim_{s\to 0^+} \frac{C(n,s)}{s(1-s)}= \frac{2}{\omega_{n-1}}$.
\end{itemize}
where $\omega_{n-1}$ denotes the $(n-1)$-dimensional measure of the unit sphere $S^{n-1}$. \\
\end{corollary}
\begin{proof}
For any $\theta\in\R$ such that $\theta> n-1$, let us define
$$E_n(\theta):=\int_{0}^{+\infty} \frac{\rho^{n-2}}{(1+\rho^2)^{\frac{\theta}{2}}} \;d\rho.$$
Observe that the assumption on the parameter $\theta$ ensures the convergence of the integral. Furthermore, integrating by parts we get
\begin{eqnarray}\label{p}
E_n(\theta)&=& \frac{1}{n-1} \int_{0}^{+\infty} \frac{(\rho^{n-1})'}{(1+\rho^2)^{\frac{\theta}{2}}} \;d\rho\nonumber\\
&=& \frac{\theta}{n-1} \int_{0}^{+\infty} \frac{\rho^{n}}{(1+\rho^2)^{\frac{\theta+2}{2}}} \;d\rho \nonumber\\
&=& \frac{\theta}{n-1} E_{n+2}(\theta+2).
\end{eqnarray}
Then, we set
\begin{equation*}
I_n^{(1)}:=E_n(n+2)=\int_{0}^{+\infty} \frac{\rho^{n-2}}{(1+\rho^2)^{\frac{n}{2}+1}} \;d\rho
\end{equation*}
and
\begin{equation*}
I_n^{(0)}:=E_n(n)=\int_{0}^{+\infty} \frac{\rho^{n-2}}{(1+\rho^2)^{\frac{n}{2}}} \;d\rho.
\end{equation*}
In view of \eqref{p}, it follows that $I_n^{(1)}$ and $I_n^{(0)}$ can be obtained in a recursive way, since
\begin{equation}\label{p1}
I_{n+2}^{(1)}=E_{n+2}(n+4)= \frac{n-1}{n+2}\, E_n(n+2)=\frac{n-1}{n+2}\, I_n^{(1)}
\end{equation}
and
\begin{equation}\label{p2}
I_{n+2}^{(0)}=E_{n+2}(n+2)= \frac{n-1}{n}\, E_n(n)=\frac{n-1}{n}\, I_n^{(0)}.
\end{equation}
\vspace{1mm}

Now we claim that 
\begin{equation}\label{p3}
I_{n}^{(1)}=\frac{\omega_{n-1}}{2n\omega_{n-2}}
\end{equation}
and
\begin{equation}\label{p4}
I_{n}^{(0)}=\frac{\omega_{n-1}}{2\omega_{n-2}}.
\end{equation}
We will prove the previous identities by induction. We start by noticing that the inductive basis are satisfied, since 
$$ I_2^{(1)}= \int_{0}^{+\infty} \frac{1}{(1+\rho^2)^{2}} \;d\rho=\frac{\pi}{4},\qquad \quad I_3^{(1)}= \int_{0}^{+\infty} \frac{\rho}{(1+\rho^2)^{\frac{5}{2}}} \;d\rho=\frac{1}{3}$$ and 
$$ I_2^{(0)}= \int_{0}^{+\infty} \frac{1}{(1+\rho^2)} \;d\rho=\frac{\pi}{2},\qquad \quad I_3^{(0)}= \int_{0}^{+\infty} \frac{\rho}{(1+\rho^2)^{\frac{3}{2}}} \;d\rho =1.$$ 
Now, using \eqref{p1} and \eqref{p2}, respectively, it is clear that in order to check the inductive steps, it suffices to verify that 
\begin{equation}\label{mis}
\frac{\omega_{n+1}}{\omega_n}=\frac{n-1}{n} \frac{\omega_{n-1}}{\omega_{n-2}}.
\end{equation}
We claim that the above formula plainly  follows  from a classical recursive formula on~$\omega_n$,
that is
\begin{equation}\label{Rec}
\omega_n = \frac{2\pi}{n-1} \omega_{n-2}.
\end{equation}
To prove this, let us denote by $\varpi_n$ the Lebesgue measure
of the $n$-dimensional unit ball and let
us fix the notation~$x=(\tilde x,x')\in\R^{n-2}\times\R^2$. By   integrating on~$\R^{n-2}$ and then using polar coordinates in~$\R^2$, we see that
\begin{eqnarray}\label{RR1}
\varpi_n &=&\int_{|x|^2\le 1}\,dx \;=\;
\int_{|x'|\le 1}\left( \int_{|\tilde x|^2 \le 1-|x'|^2}\,d\tilde x\right)dx' \nonumber \\ \nonumber \\
&=& \varpi_{n-2} \int_{|x'|\le 1} \big(1-|x'|^2\big)^{\frac{(n-2)}{2}}\,dx'  \nonumber \\ \nonumber \\
&=& 2\pi \varpi_{n-2} \int_{0}^1 \rho \,\big(1-\rho^2\big)^{\frac{(n-2)}{2}}\,d\rho \ = \
 \frac{2\pi\varpi_{n-2}}{n}.
\end{eqnarray}
Moreover, by polar coordinates in~$\R^n$,
\begin{equation}\label{RR2}
\varpi_n = \int_{|x|\le1}\,dx=\omega_{n-1}\int_0^1\rho^{n-1}\,d\rho
=\frac{\omega_{n-1}}{n}.
\end{equation}
Thus, we use \eqref{RR2}
and \eqref{RR1} and we obtain
$$ \omega_{n-1}=n\varpi_n={2\pi\varpi_{n-2}} =
\frac{2\pi\omega_{n-3}}{n-2},$$
which is \eqref{Rec}, up to replacing $n$ with~$n-1$. In turn, \eqref{Rec} implies \eqref{mis} and so \eqref{p3} and \eqref{p4}.

\vspace{1mm}

Finally, using \eqref{p3}, \eqref{p4} and Proposition \ref{asy_est} we can conclude that
$$\lim_{s\to 1^{-}} \frac{C(n,s)}{s(1-s)}= \frac{2}{\omega_{n-2} I_n^{(1)}}= \frac{4n}{\omega_{n-1}} $$
and
$$\lim_{s\to 0^{+}} \frac{C(n,s)}{s(1-s)}= \frac{1}{\omega_{n-2} I_n^{(0)}}= \frac{2}{\omega_{n-1}}, $$
as desired\footnote{Another (less elementary) way to obtain this result is to notice that $\dys E_n(\theta)=2\,\mathcal{B} \left( (n-1)/2, (\theta-n-1)/2 \right)$, where $\mathcal{B}$ is the Beta function.}.
\end{proof} 

\vspace{2mm}
\begin{rem}
It is worth noticing that when $p=2$ we recover the constants $C_1$ and $C_2$ in \eqref{disBre} and \eqref{disMaz}, respectively. In fact, in this case it is known that $$C_1=\frac{1}{2}\int_{S^{n-1}}|\xi_1|^2 d\sigma(\xi)=\frac{1}{2n}\sum_{i=1}^n \int_{S^{n-1}}|\xi_i|^2 d\sigma(\xi) =\frac{\omega_{n-1}}{2n}$$ and $C_2=\omega_{n-1}$ (see \cite{BBM01} and \cite{MS02}). Then, by Proposition~\ref{aaaeleono} and Corollary~\ref{asy_est1} it follows that
\begin{eqnarray}\
&& \lim_{s\to 1^-} (1-s) \int_{\R^n}\int_{\R^n} \frac{|u(x)-u(y)|^2}{|x-y|^{n+2s}}\,dx\,dy \nonumber\\
\nonumber\\
&&\qquad \qquad\qquad\qquad\quad = \lim_{s\to 1^-} 2(1-s) C(n,s)^{-1}\||\xi|^s\FF u \|^2_{L^2(\R^n)}\nonumber \\
\nonumber\\
&&\qquad \qquad\qquad\qquad \quad = \frac{\omega_{n-1}}{2n} \| \nabla u\|^2_{L^2(\R^n)} \nonumber\\
\nonumber \\
&& \qquad \qquad\qquad\qquad \quad = C_1 \|u\|^2_{H^1(\R^n)} \nonumber
\end{eqnarray}
and
\begin{eqnarray}\
\lim_{s\to 0^+} s \int_{\R^n}\int_{\R^n} \frac{|u(x)-u(y)|^2}{|x-y|^{n+2s}}\,dx\,dy &=& \lim_{s\to 0^+} 2s C(n,s)^{-1}\||\xi|^s\FF u \|^2_{L^2(\R^n)}\nonumber \\
\nonumber\\
&=& \omega_{n-1} \|u\|^2_{L^2(\R^n)}\nonumber \\
\nonumber\\
&=& C_2 \|u\|^2_{L^2(\R^n)}.\nonumber
\end{eqnarray}
\end{rem}
\vspace{2mm}
We will conclude this section with the following proposition that one could  plainly deduce from Proposition \ref{pro_symbol}. We prefer to provide a direct proof, based on Lemma~\ref{lem_2nd}, in order to show the consistency in the definition of the constant $C(n,s)$.

\begin{prop}
Let $n>1$. For any $u\in C_0^{\infty}(\R^n)$ the following statements hold:
\begin{itemize}
\item[$(i)$] $\lim_{s\to 0^+} (-\Delta)^s u=u$;
\item[$(ii)$] $\lim_{s\to 1^-} (-\Delta)^s u=-\Delta u$.
\end{itemize}
\end{prop}
\begin{proof}
Fix $x\in \R^n$, $R_0>0$ such that $\textrm{supp}\, u \subseteq B_{R_0}$ and set $R= R_0 + |x|+1$. First,\begin{eqnarray}\label{el0}
\left|   \int_{B_R} \frac{u(x+y)+u(x-y)-2u(x)}{|y|^{n+2s}}\,dy  \right| 
& \leq &  \|u\|_{C^2(\R^n)}\int_{ B_R} \frac{|y|^2}{|y|^{n+2s}}\,dy\nonumber \\
 \nonumber \\
&\leq & \omega_{n-1} \|u\|_{C^2(\R^n)} \int_{0}^R \frac{1}{\rho^{2s-1}}\,d\rho \nonumber \\
 \nonumber \\
&= & \frac{\omega_{n-1} \|u\|_{C^2(\R^n)}R^{2-2s}}{2(1-s)}.
\end{eqnarray}
\vspace{2mm}

Furthermore, observe that  $|y|\geq R$ yields $|x\pm y|\geq |y|-|x|\geq R-|x|>R_0$ and consequently $u(x\pm y)=0$. Therefore,
\begin{eqnarray}\label{el1}
-\frac{1}{2} \int_{\R^n \setminus B_R}  \frac{u(x+y)+u(x-y)-2u(x)}{|y|^{n+2s}}\;dy & = & u(x) \int_{\R^n \setminus B_R} \frac{1}{|y|^{n+2s}}\,dy \nonumber\\
& =& \omega_{n-1} u(x) \int_{R}^{+\infty} \frac{1}{\rho^{2s+1}} \,d\rho \nonumber \\
& = & \frac{\omega_{n-1} R^{-2s}}{2s} u(x).
\end{eqnarray}
Now, by \eqref{el0} and Corollary \ref{asy_est1}, we have
$$\lim_{s\to 0^+} -\frac{C(n,s)}{2}\int_{B_R} \frac{u(x+y)+u(x-y)-2u(x)}{|y|^{n+2s}}\;dy =0 $$ and so we get, recalling Lemma \ref{lem_2nd},
\begin{eqnarray*}
\lim_{s\to 0^+} (-\Delta)^s u &=& \lim_{s\to 0^+} -\frac{C(n,s)}{2}\int_{\R^n \setminus B_R} \frac{u(x+y)+u(x-y)-2u(x)}{|y|^{n+2s}}\;dy \\
&=& \lim_{s\to 0^+} \frac{C(n,s)\omega_{n-1} R^{-2s}}{2s} u(x)\;= \;u(x),
\end{eqnarray*}
where the last identities follow from \eqref{el1} and again Corollary \ref{asy_est1}. This proves~$(i)$.

\vspace{2mm}

Similarly, we can prove $(ii)$. In this case, when $s$ goes to 1, we have no contribution outside the unit ball, as the following estimate shows
\begin{eqnarray}\label{el3}
&& \left| \int_{\R^n\setminus B_1} \frac{u(x+y)+u(x-y)-2u(x)}{|y|^{n+2s}}\;dy  \right| \nonumber \\
&&\qquad\qquad\qquad\qquad\qquad\qquad\leq  4 \|u\|_{L^{\infty}(\R^n)}\int_{ \R^n\setminus B_1} \frac{1}{|y|^{n+2s}}\;dy\nonumber \\
&&\qquad\qquad\qquad\qquad\qquad\qquad\leq  4 \omega_{n-1} \|u\|_{L^{\infty}(\R^n)} \int_{1}^{+\infty} \frac{1}{\rho^{2s+1}}\,d\rho \nonumber \\
&&\qquad\qquad\qquad\qquad\qquad\qquad= \frac{2\omega_{n-1}}{s} \|u\|_{L^{\infty}(\R^n)}.\nonumber
\end{eqnarray}
As a consequence (recalling Corollary \ref{asy_est1}), we get
\begin{equation}\label{el3a}
\lim_{s\to 1^-} -\frac{C(n,s)}{2}\int_{\R^n\setminus B_1} \frac{u(x+y)+u(x-y)-2u(x)}{|y|^{n+2s}}\;dy =0.
\end{equation}
\vspace{2mm}

On the other hand, we have
\begin{eqnarray*}
&& \left|   \int_{B_1} \frac{u(x+y)+u(x-y)-2u(x)- D^2u(x)y\cdot y}{|y|^{n+2s}}\;dy  \right| \nonumber \\
&&\qquad \qquad \qquad \qquad \qquad \qquad \qquad \qquad \ \leq  \|u\|_{C^3(\R^n)}\int_{ B_1} \frac{|y|^3}{|y|^{n+2s}}\;dy\nonumber \\
&&\qquad \qquad \qquad \qquad \qquad \qquad \qquad \qquad \ \leq  \omega_{n-1} \|u\|_{C^3(\R^n)} \int_{0}^1 \frac{1}{\rho^{2s-2}}\, d\rho \nonumber \\
\\ \nonumber
&&\qquad \qquad \qquad \qquad \qquad \qquad \qquad \qquad \ = \frac{\omega_{n-1} \|u\|_{C^3(\R^n)}}{3-2s}\nonumber
\end{eqnarray*}
and this implies that
\begin{eqnarray}\label{el4a}
&& \lim_{s\to 1^-} -\frac{C(n,s)}{2}\int_{ B_1} \frac{u(x+y)+u(x-y)-2u(x)}{|y|^{n+2s}}\;dy\; \nonumber \\
&& \qquad \qquad \qquad \qquad \qquad  = \lim_{s\to 1^-} -\frac{C(n,s)}{2}\int_{ B_1} \frac{D^2u(x)y\cdot y}{|y|^{n+2s}}\;dy.
\end{eqnarray}
\vspace{1mm}

Now, notice that if $i\neq j$ then
$$ \int_{ B_1} \partial^2_{ij}u(x) y_i\cdot y_j\;dy \;=\; -\int_{ B_1} \partial^2_{ij} u(x)\tilde{y}_i \cdot \tilde{y}_j \;d\tilde{y},
$$
where $\tilde{y}_k=y_k$ for any $k\neq j$ and $\tilde{y}_j=-y_j$, and thus
\begin{equation}\label{el5}
\int_{ B_1} \partial^2_{ij}u(x)y_i\cdot y_j  \;dy=0.
\end{equation}
Also, up to permutations, for any fixed $i$, we get
\begin{eqnarray}\label{el6}
\int_{ B_1} \frac{  \partial^2_{ii}u(x)y_i^2}{|y|^{n+2s}}\;dy &=& \partial^2_{ii}u(x) \int_{ B_1} \frac{y_i^2}{|y|^{n+2s}}\;dy \ = \ \partial^2_{ii}u(x) \int_{ B_1} \frac{y_1^2}{|y|^{n+2s}}\;dy \nonumber\\
&=& \frac{\partial^2_{ii}u(x)}{n}\sum_{j=1}^n \int_{ B_1} \frac{y_j^2}{|y|^{n+2s}}\;dy \ = \ \frac{ \partial^2_{ii}u(x)}{n} \int_{ B_1} \frac{|y|^2}{|y|^{n+2s}}\;dy \nonumber\\
&=& \frac{ \partial^2_{ii}u(x)\; \omega_{n-1}}{2n(1-s)}. 
\end{eqnarray}

Finally, combining \eqref{el3a}, \eqref{el4a}, \eqref{el5}, \eqref{el6}, Lemma~\ref{lem_2nd} and Corollary~\ref{asy_est1}, we can conclude
\begin{eqnarray*}
\lim_{s\to 1^-} (-\Delta)^s u &=&\lim_{s\to 1^-} -\frac{C(n,s)}{2}\int_{B_1} \frac{u(x+y)+u(x-y)-2u(x)}{|y|^{n+2s}}\;dy \\
&=&\lim_{s\to 1^-} -\frac{C(n,s)}{2}\int_{B_1} \frac{D^2u(x) y\cdot y}{|y|^{n+2s}}\;dy \\
&=&\lim_{s\to 1^-} -\frac{C(n,s)}{2}\sum_{i=1}^n \int_{B_1} \frac{\partial^2_{ii}u(x) y_i^2}{|y|^{n+2s}}\;dy \\
&=& \lim_{s\to 1^-} -\frac{C(n,s)  \omega_{n-1}}{4n(1-s)}\;\sum_{i=1}^n  \partial^2_{ii}u(x)\;=\; -\Delta u(x).\qedhere
\end{eqnarray*}

\end{proof}

\vspace{3mm}

\section{Extending a $W^{s,p}(\Omega)$ function to the whole of $\R^n$}\label{sec_estensione}
As well known when $s$ is an integer, under certain regularity assumptions on the domain $\Omega$, any function in $W^{s,p}(\Omega)$ may be extended to a function in $W^{s,p}(\R^n)$. Extension results are quite important in applications and are necessary in order to improve some embeddings theorems, in the classic case as well as in the fractional case (see Section~\ref{sec_sobolev} and Section~\ref{sec_compactness} in the following). 

\vspace{2mm}

For any $s\in(0,1)$ and any $p\in[1,\infty)$, we say that an open set
$\Omega\subseteq\R^n$ is an {\it extension domain for~$W^{s,p}$} if
there exists a positive constant $C=C(n,p,s,\Omega)$ such that:
for every function $u\in W^{s,p}(\Omega)$
there exists $\tilde{u}\in W^{s,p}(\R^n)$ with
$\tilde u(x)=u(x)$ for all $x\in\Omega$ and
$\dys\|\tilde{u}\|_{W^{s,p}(\R^n)}\leq C\|u\|_{W^{s,p}(\Omega)}$.
\vspace{2mm}

In general, an arbitrary open set is not an extension domain for $W^{s,p}$.
To~the~authors' knowledge, the problem of characterizing  the class of sets that are extension domains for $W^{s,p}$ is open\footnote{
While revising this paper, we were informed that an answer to this question has been given by Zhou, by analyzing the link between extension domains in $W^{s,p}$ and the measure density condition (see~\cite{Zho11}).
}. When $s$ is an integer, we cite~\cite{Jon82} for a complete characterization in the special case $s=1$, $p=2$ and $n=2$, and we refer the interested reader to the recent book by Leoni~\cite{Leo09}, in which this problem is very well discussed (see, in particular, Chapter~11 and Chapter~12 there).

\vspace{2mm}

In this section, we will show that any open set~$\Omega$ of class~$C^{0,1}$ with bounded boundary is an extension domain for~$W^{s,p}$.

\vspace{2mm}

We start with some preliminary lemmas, in which we will construct the extension to the whole of $\R^n$ of a function $u$ defined 
 on $\Omega$ in two separated cases: when the function $u$ is identically zero in a neighborhood of the boundary $\partial \Omega$ and when $\Omega$ coincides with the half-space $\R^n_+$.
\vspace{1mm}

\begin{lemma}\label{giampi0}
Let $\Omega$ be an open set in $\R^n$ and $u$ a function in $W^{s,p}(\Omega)$ with $s\in(0,1)$ and $ p \in [1,+\infty)$. If there exists a compact subset $K\subset\Omega$ such that $u\equiv 0$ in $\Omega\setminus K$, then the extension function $\tilde{u}$ defined as
\begin{equation}\label{giampieq0}
\tilde{u}(x)=
\begin{cases}
  u(x) \quad x\in \Omega \,,\\
\,0 \qquad x\in \R^n \setminus \Omega\,
\end{cases}
\end{equation}
belongs to $W^{s,p}(\R^n)$ and 
$$
\|\tilde{u}\|_{W^{s,p}(\R^n)} \leq C \|u\|_{W^{s,p}(\Omega)},
$$
where $C$ is a suitable positive constant depending on $n$, $p$, $s$, $K$ and $\Omega$.
\end{lemma}

\begin{proof}
Clearly $\tilde{u}\in L^p(\R^n)$. Hence, it remains to verify that the Gagliardo norm of $\tilde{u}$ in $\R^n$ is bounded by the one of $u$ in $\Omega$. Using the symmetry of the integral in the Gagliardo norm with respect to $x$ and $y$ and the fact that $\tilde{u}\equiv0$ in $\R^n\setminus \Omega$, we can split as follows
\begin{eqnarray}\label{est0}
\int_{\R^n}\int_{\R^n} \frac{|\tilde{u}(x)-\tilde{u}(y)|^p}{|x-y|^{n+sp}} \,dx\,dy &\! = &\! \int_{\Omega}\int_{\Omega} \frac{|u(x)-u(y)|^p}{|x-y|^{n+sp}} \,dx\,dy  \\
&&\! + 2 \int_{\Omega} \left( \int_{\R^n \setminus \Omega} \frac{|u(x)|^p}{|x-y|^{n+sp}} \,dy \right)\,dx,  \nonumber
\end{eqnarray}
where the first term in the right hand-side of~\eqref{est0} is finite since $u\in W^{s,p}(\Omega)$. Furthermore, for any $y\in\R^n\setminus K$, 
$$ \frac{|u(x)|^p}{|x-y|^{n+sp}}= \frac{\chi_{K} (x) |u(x)|^p}{|x-y|^{n+sp}}\,\leq\, \chi_{K} (x) |u(x)|^p\, \sup_{x\in K} \frac{1}{|x-y|^{n+sp}} $$
and so
\begin{equation}\label{est1}
\int_{\Omega} \left( \int_{\R^n\setminus \Omega} \frac{|u(x)|^p}{|x-y|^{n+sp}} \,dy\right)\,dx \leq \int_{\R^n\setminus\Omega}\, \frac{1}{{\rm dist(}y,\partial{K})^{n+sp}}\,dy \,\;\|u\|^p_{	L^{p}(\Omega)}\,.
\end{equation}
 Note that the integral in \eqref{est1} is finite since ${\rm dist(}\partial{\Omega},\partial{K})\geq \alpha>0$ and $n+sp>n$. Combining \eqref{est0} with \eqref{est1}, we get
$$\|\tilde{u}\|_{W^{s,p}(\R^n)} \leq C \|u\|_{W^{s,p}(\Omega)}    $$
where $C=C(n,s,p,K)$.
\end{proof}

\vspace{1mm}

\begin{lemma}\label{giampi1}
Let $\Omega$ be an open set in $\R^n$, symmetric with respect to the coordinate $x_n$, and consider the sets $\Omega_+=\{x\in \Omega \,:\, x_n > 0\}$ and $\Omega_-=\{x\in \Omega \,:\, x_n\leq 0\}$. Let $u$ be a function in $W^{s,p}(\Omega_+)$, with $s\in(0,1)$ and $ p \in [1, +\infty)$. Define
\begin{equation}\label{giampieq1}
\bar{u}(x)=
\begin{cases}
u(x',x_n)\quad\quad x_n\geq 0\,, \\
u(x',-x_n) \;\quad x_n<0 \,.
\end{cases}
\end{equation}
Then $\bar{u}$ belongs to $W^{s,p}(\Omega)$ and
$$
\|\bar{u}\|_{W^{s,p}(\Omega)} \leq 4 \|u\|_{W^{s,p}(\Omega_+)}\,.
$$
\end{lemma}

\begin{proof}
By splitting the integrals and changing variable $\hat{x}= (x',-x_n)$, we get
\begin{equation}\label{giampieq2}
\|\bar{u}\|^p_{L^p(\Omega)}\,=\, \int_{\Omega_+} |u(x)|^p \,dx + \int_{\Omega_+} |u(\hat{x}',\hat{x_n} )|^p \,d\tilde{x}\,=\, 2 \|u\|^p_{L^p(\Omega_+)}.
\end{equation}
Also, if $x\in \R^n_+$ and $y\in \CC\R^n_+$ then $(x_n-y_n)^2\geq (x_n+y_n)^2 $ and therefore
\begin{eqnarray*}
\int_{\Omega}\int_{\Omega} \frac{|\bar{u}(x)-\bar{u}(y)|^p}{|x-y|^{n+sp}} \,dx\,dy \! &=&\! \int_{\Omega_+}\int_{\Omega_+} \frac{|u(x)-u(y)|^p}{|x-y|^{n+sp}} \,dx\,dy \\
&& \!+ 2\int_{\Omega_+}\int_{\CC\Omega_+} \frac{|u(x)-u(y',-y_n)|^p}{|x-y|^{n+sp}} \,dx\,dy \\
&& \! + \int_{\CC\Omega_+}\int_{\CC\Omega_+} \frac{|u(x',-x_n)-u(y',-y_n)|^p}{|x-y|^{n+sp}} \,dx\,dy\\
\\
& \leq & \! 4\|u\|^p_{W^{s,p}(\Omega_+)}.
\end{eqnarray*}
This concludes the proof.
\end{proof}
\vspace{1mm}

Now, a truncation lemma near $\partial{\Omega}$.
\begin{lemma}\label{giampi3}
Let $\Omega$ be an open set in $\R^n$, $s\in(0,1)$ and $p \in [1, +\infty)$. Let us consider $u\in W^{s,p}(\Omega)$ and $\psi\in C^{0,1}(\Omega)$, $0\leq \psi\leq 1$. Then $\psi u\in W^{s,p}(\Omega)$ and
\begin{equation}\label{enrico56}
\|\psi\, u\|_{W^{s,p}(\Omega)} \leq C\|u\|_{W^{s,p}(\Omega)},
\end{equation}
where $C=C(n,p,s,\Omega)$.
\end{lemma}

\begin{proof}
It is clear that $\|\psi\,u\|_{L^p(\Omega)} \leq \|u\|_{L^p(\Omega)}$ since $|\psi|\leq 1$. Furthermore, adding and subtracting the factor $\psi(x)u(y)$, we get
\begin{eqnarray}\label{giampieq3}
&& \int_{\Omega}\int_{\Omega} \frac{|\psi(x)\,u(x)-\psi(y)\,u(y)|^p}{|x-y|^{n+sp}} \,dx\,dy \nonumber \\
&&\qquad \qquad \qquad \quad  \leq \, 2^{p-1}\bigg( \int_{\Omega}\int_{\Omega} \frac{|\psi(x)\,u(x)-\psi(x)\,u(y)|^p}{|x-y|^{n+sp}} \,dx\,dy \nonumber \\
&&\qquad \qquad \qquad  \quad \quad + \int_{\Omega}\int_{\Omega} \frac{|\psi(x)\,u(y)-\psi(y)\,u(y)|^p}{|x-y|^{n+sp}} \,dx\,dy\bigg) \nonumber \\
\nonumber \\
&&\qquad \qquad \qquad\quad  \leq \,  2^{p-1}\bigg( \int_{\Omega}\int_{\Omega} \frac{|u(x)-u(y)|^p}{|x-y|^{n+sp}} \,dx\,dy \\
&&\qquad \qquad \qquad  \quad \quad + \int_{\Omega}\int_{\Omega} \frac{|u(x)|^p\,|\psi(x)-\psi(y)|^p}{|x-y|^{n+sp}} \,dx\,dy\bigg).\nonumber
\end{eqnarray}
Since $\psi$ belongs to $C^{0,1}(\Omega)$, we have
\begin{eqnarray}\label{giampieq4}
\int_{\Omega}\int_{\Omega} \frac{|u(x)|^p\,|\psi(x)-\psi(y)|^p}{|x-y|^{n+sp}} \,dx\,dy \! &\leq & \! \Lambda^p\! \int_{\Omega}\int_{\Omega\cap |x-y|\leq 1} \frac{|u(x)|^p\,|x-y|^p}{|x-y|^{n+sp}} \,dx\,dy \nonumber \\ 
&&+ \int_{\Omega}\int_{\Omega \cap |x-y|\geq 1} \frac{|u(x)|^p}{|x-y|^{n+sp}} \,dx\,dy \nonumber \\ 
\nonumber \\
&\leq & \! \tilde{C} \|u\|^p_{L^p(\Omega)},
\end{eqnarray}
where $\Lambda$ denotes the Lipschitz constant of $\psi$ and $\tilde{C}$ is a positive constant depending on $n$, $p$ and $s$. Note that the last inequality follows from the fact that the kernel $|x-y|^{-n+(1-s)p}$ is summable with respect to $y$ if $|x-y|\leq 1$ since $n+(s-1)p<n$ and, on the other hand, the kernel $|x-y|^{-n-sp}$ is summable when $|x-y|\geq 1$ since $n+sp>n$. Finally, combining \eqref{giampieq3} with \eqref{giampieq4}, we obtain estimate \eqref{enrico56}.
\end{proof}
\vspace{1mm}

Now, we are ready to prove the main theorem of this section, that states that every open Lipschitz set~$\Omega$ with bounded boundary is an extension domain for~$W^{s,p}$.
\begin{thm}\label{thm_estensione}
Let $p\in [1,+\infty)$, $s\in(0,1)$ and $\Omega \subseteq \R^n$ be an open set of class $C^{0,1}$ with bounded boundary\footnote{ Motivated by an interesting remark of the anonymous Referee, we point out that it should be expected that the Lipschitz assumption on the boundary of~$\Omega$ may be weakened when~$s\in(0,1)$, since in the case~$s=0$ clearly no regularity at all is needed for the extension problem.}. Then $W^{s,p}(\Omega)$ is continuously embedded in $W^{s,p}(\R^n)$, namely for any $u\in W^{s,p}(\Omega)$ there exists $\tilde{u}\in W^{s,p}(\R^n)$ such that $\tilde{u}|_{\Omega}=u$ and 
$$
\|\tilde{u}\|_{W^{s,p}(\R^n)} \leq C \|u\|_{W^{s,p}(\Omega)}
$$
where $C=C(n,p,s,\Omega)$.
\end{thm}
\begin{proof}
Since $\partial{\Omega}$ is compact, we can find a finite number of balls $B_j$ such that $\dys \partial{\Omega} \subset \bigcup_{j=1}^k B_j$ and so we can write $\dys \R^n= \bigcup_{j=1}^k B_j \cup \left( \R^n \setminus \partial{\Omega}  \right)$.
\vspace{1mm}

If we consider this covering, there exists a partition of unity  related to it, i.e. there exist $k+1$ smooth functions $\psi_0$, $\psi_1$,..., $\psi_k$ such that $\text{\rm spt}\,\psi_0\subset \R^n\setminus \partial{\Omega}$, $\text{\rm spt}\,\psi_j\subset B_j$ for any $j\in \{1,...,k\}$, $0\leq \psi_j \leq 1$ for any $j\in \{0,... ,k\}$ and $\dys \sum_{j=0}^k \psi_j=1$. Clearly, 
$$
u=\sum_{j=0}^k \psi_j u\,.
$$
By Lemma \ref{giampi3}, we know that $\psi_0\,u $ belongs to $W^{s,p}(\Omega)$. Furthermore, since $\psi_0\,u\equiv 0$ in a neighborhood of $\partial \Omega$, we can extend it to the whole of $\R^n$, by setting
$$
\widetilde{\psi_0\,u}(x)=
\begin{cases}
\psi_0\,u(x) & x\in\Omega, \\
\,\;0 & x\in \R^n\setminus \Omega
\end{cases}
$$ 
and $\widetilde{\psi_0\,u}\in W^{s,p}(\R^n)$. Precisely
\begin{equation}\label{e1}
\|\widetilde{\psi_0\,u}\|_{W^{s,p}(\R^n)}\,\leq \,C\,\| \psi_0\,u\|_{W^{s,p}(\Omega)}\,\leq\, C\, \|u\|_{W^{s,p}(\Omega)},
\end{equation}
where $C=C(n,s,p,\Omega)$ (possibly different step by step, see Lemma~\ref{giampi0} and Lemma~\ref{giampi3}). 
\vspace{2mm}

For any $j\in\{1,...,k\}$, let us consider $u|_{B_j\cap\Omega }$ and set 
$$
v_j(y):=u\left( T_j(y) \right) \qquad \mbox{for any}\ y\in Q_+,
$$
where $T_j: Q\to B_j$ is the isomorphism of class $C^{0,1}$ defined in Section~\ref{definizioni}. Note that such a $T_j$ exists by the regularity assumption on the domain $\Omega$.

\vspace{2mm}

Now, we state that $v_j \in W^{s,p}\left( Q_+ \right)$. Indeed, using the standard changing variable formula by setting $x=T_j(\hat{x})$ we have
\begin{eqnarray}\label{esteng}
&& \int_{Q_+}\int_{Q_+} \frac{|v(\hat{x})-v(\hat{y})|^p}{|\hat{x}-\hat{y}|^{n+sp}}\,d\hat{x}\,d\hat{y} \nonumber \\ 
&& \qquad \qquad \qquad \qquad= \ \int_{Q_+}\int_{Q_+} \frac{|u(T_j(\hat{x}))-u(T_j(\hat{y}))|^p}{|\hat{x}-\hat{y}|^{n+sp}}\,d\hat{x}\,d\hat{y} \nonumber \\
\nonumber \\
 && \qquad \qquad \qquad  \qquad= \  \int_{B_j\cap\Omega}\int_{B_j\cap \Omega} \frac{|u(x)-u(y)|^p}{|T_j^{-1}(x)-T_j^{-1}(y)|^{n+sp}}\,\det(T_j^{-1})dx\,dy \nonumber \\
 \nonumber \\
 && \qquad \qquad \qquad \qquad \leq \ C \int_{B_j\cap\Omega}\int_{B_j\cap \Omega} \frac{|u(x)-u(y)|^p}{|x-y|^{n+sp}}\,dx\,dy,
\end{eqnarray}
where \eqref{esteng} follows from the fact that $T_j$ is bi-Lipschitz.
Moreover, using Lemma \ref{giampi1} we can extend $v_j$ to all $Q$ so that the extension $\bar{v}_j$ belongs to $W^{s,p}(Q)$ and 
$$
\| \bar{v}_j\|_{W^{s,p}(Q)}\leq 4 \|v_j\|_{W^{s,p}(Q_+)}.
$$
\vspace{2mm}

We set 
$$
w_j(x):=\bar{v}_j\left( T^{-1}_j(x) \right) \ \ \mbox{for any}\ x\in B_j.  
$$
\vspace{1mm}

Since $T_j$ is bi-Lipschitz, by arguing as above it follows that $w_j\in W^{s,p}(B_j)$. Note that  $w_j\equiv u$ (and consequently $\psi_j\,w_j\equiv \psi_j \,u$) on $B_j\cap \Omega$. By definition $\psi_j\,w_j$ has compact support in $B_j$ and therefore, as done for $\psi_0\,u$, we can consider the extension $\widetilde{\psi_j\,w_j}$ to all $\R^n$ in such a way that $\widetilde{\psi_j\, w_j} \in W^{s,p}(\R^n)$. Also, using Lemma~\ref{giampi0}, Lemma~\ref{giampi1}, Lemma~\ref{giampi3} and estimate \eqref{esteng} we get
\begin{eqnarray}\label{e0}
\|\widetilde{\psi_j\,w_j}\|_{W^{s,p}(\R^n)}&\leq&  C \|\psi_j\,w_j\|_{W^{s,p}(B_j)} \, \leq C \,  \|w_j\|_{W^{s,p}(B_j)}\nonumber \\[1ex]
& \leq & C \|\bar{v}_j\|_{W^{s,p}(Q)} \,  \leq  \,  C \|v_j\|_{W^{s,p}(Q_+)}\nonumber \\[1ex]
& \leq  & C \|u\|_{W^{s,p}(\Omega\cap B_j)},
\end{eqnarray}
where $C=C(n,p,s,\Omega)$ and it is possibly different step by step.

\vspace{2mm}

Finally, let
$$
\tilde{u}=\widetilde{\psi_0\,u}+ \sum_{j=1}^k \widetilde{\psi_j \,w_j}
$$
be the extension of $u$ defined on all $\R^n$. By construction, it is clear that $\tilde{u}|_{\Omega}= u$ and, combining \eqref{e1} with \eqref{e0}, we get
$$
\|\tilde{u}\|_{W^{s,p}(\R^n)} \, \leq \,  C \|u\|_{W^{s,p}(\Omega)}
$$
with  $C=C(n,p,s,\Omega)$.
\end{proof}

\vspace{1mm}

\begin{corollary}
Let $p \in [1, +\infty)$, $s\in(0,1)$ and $\Omega $ be an open set in $\R^n$ of class $C^{0,1}$ with bounded boundary. Then for any $u\in W^{s,p}(\Omega)$, there exists a sequence $\{u_n\}\in C^{\infty}_0(\R^n)$ such that $u_n \to u$ as $n\rightarrow +\infty$ in $W^{s,p}(\Omega)$, i.e.,
$$\lim_{n\rightarrow +\infty} \|u_n-u\|_{W^{s,p}(\Omega)}=0\,.$$
\end{corollary}

\begin{proof}
The proof follows directly by Theorem~\ref{density} and Theorem~\ref{thm_estensione}.
\end{proof}

\vspace{3mm}
\section{Fractional Sobolev inequalities}\label{sec_sobolev}

In this section, we provide an elementary proof of a
Sobolev-type inequality involving the fractional norm $\|\cdot\|_{W^{s,p}}$ (see~Theorem~\ref{thm_sobolev} below).

The original proof is contained in the Appendix of~\cite{SV11} and it deals with the case $p=2$ (see, in particular, Theorem 7 there).  We note that when~$p\!\!~=~\!\!2$ and $s\in[1/2,1)$ some of the statements may be strengthened (see~\cite{Bre02}).
We also note that more general embeddings for the spaces $W^{s,p}$ can be obtained by interpolation techniques and by passing through Besov  spaces; see, for instance, \cite{Bes59, Bes59b, Usp60, Usp60b, Liz60}. For a more comprehensive treatment of fractional Sobolev-type inequalities we refer to \cite{Lio63, LM68, BM01, Ada75, Tar07} and the references therein.

We remark that the proof here is self-contained. Moreover, we will not make use of
Besov or fancy interpolation spaces.

\vspace{2mm}

In order to prove the Sobolev-type inequality in forthcoming Theorem~\ref{thm_sobolev}, we need some preliminary results. The first of them is an elementary estimate involving the measure of finite measurable sets $E$ in $\R^n$ as stated in the following lemma (see~\cite[Lemma A.1]{SV10} and also~\cite[Corollary~24 and~25]{CV11}).

\begin{lemma}\label{5yhh}
Fix $x\in\R^n$. Let $p\in [1,+\infty)$, $s\in(0,1)$ and $E\subset\R^n$ be a measurable set with finite
measure.
Then,
$$
\int_{\CC E}\frac{dy}{|x-y|^{n+sp}}\ge C\,|E|^{-sp/n},
$$
for a suitable constant $C=C(n,p,s)>0$.
\end{lemma}
\begin{proof}
We set
$$
\rho:=\left( \frac{|E|}{\omega_n}\right)^{\!\frac{1}{n}}
$$
and then it follows
\begin{eqnarray*}
|(\CC E)\cap B_{\rho}(x)| \! & = & \! |B_{\rho}(x)|-|E\cap B_{\rho}(x)| \, = \, |E|-|E\cap B_{\rho}(x)| \\
& = & \! |E\cap\CC B_{\rho}(x)|.
\end{eqnarray*}
Therefore,
\begin{eqnarray*}
\int_{\CC E} \frac{dy}{|x-y|^{n+sp}} \! & = & \! \int_{(\CC E) \cap B_{\rho}(x)} \frac{dy}{|x-y|^{n+sp}}
+\int_{(\CC E)\cap \CC B_{\rho}(x)} \frac{dy}{|x-y|^{n+sp}} \\ \\
& \geq & \! \int_{(\CC E) \cap B_{\rho}(x)} \frac{dy}{\rho^{n+sp}}
+\int_{(\CC E)\cap \CC B_{\rho}(x)} \frac{dy}{|x-y|^{n+sp}}\\ \\
& = & \! \frac{|(\CC E)\cap B_{\rho}(x)|}{\rho^{n+sp}}  
+\int_{(\CC E)\cap \CC B_{\rho}(x)} \frac{dy}{|x-y|^{n+sp}} \\ \\
& = & \! \frac{|E\cap \CC B_{\rho}(x)|}{\rho^{n+sp}} 
+\int_{(\CC E)\cap \CC B_{\rho}(x)} \frac{dy}{|x-y|^{n+sp}} \\ \\
& \geq & \!  \! \int_{E \cap \CC B_{\rho}(x)} \frac{dy}{|x-y|^{n+sp}}
+\int_{(\CC E)\cap \CC B_{\rho}(x)} \frac{dy}{|x-y|^{n+sp}} \\ \\
& = & \! \int_{\CC B_{\rho}(x)} \frac{dy}{|x-y|^{n+sp}}.
\end{eqnarray*}
The desired result easily follows by using polar coordinates centered at $x$.
\end{proof}


\vspace{0.5mm}

Now, we recall a general statement about a useful summability
property (see~\cite[Lemma 5]{SV11}. For related results, see also~\cite[Lemma~4] {Dyd10}).

\begin{lemma} \label{OS1}
Let $s\in (0,1)$ and $p\in[1,+\infty)$ such that $sp<n$. Fix~$T> 1$; let $N\in\Z$ and
\begin{equation}\label{ak}\begin{split}\\ &{\mbox{
$a_k$
be a bounded, nonnegative, decreasing sequence}}\\&\qquad{\mbox{
with~$a_k=0$ for any $k\ge N$.}}
\end{split}\end{equation}
Then,
$$ \sum_{k\in\Z} a_k^{(n-sp)/n} T^{k}\le C\,
\sum_{{k\in\Z}\atop{a_k\ne0}}
a_{k+1} a_k^{-sp/n} T^{k},
$$
for a suitable constant $C=C(n,p,s,T)>0$, independent
of~$N$.\end{lemma}

\begin{proof} 
By~\eqref{ak},
\begin{equation}\label{are c}
{\mbox{both }}
\sum_{k\in\Z} a_k^{(n-sp)/n} T^{k}
{\mbox{ and }}
\sum_{{k\in\Z}\atop{a_k\ne0}}
a_{k+1} a_k^{-sp/n} T^{k}
{\mbox{ are convergent series.}}
\end{equation}
Moreover, since $a_k$ is nonnegative and
decreasing, we have that
if $a_{k}=0$, then $a_{k+1}=0$. Accordingly,
$$ \sum_{k\in\Z} a_{k+1}^{(n-sp)/n} T^{k}=
\sum_{{k\in\Z}\atop{a_k\ne 0}} a_{k+1}^{(n-sp)/n} T^{k}.$$
Therefore, we may use the H\"older
inequality with exponents $\alpha:=n/sp$ and $\beta:=n/(n-sp)$
by arguing as follows.

\begin{eqnarray*}
&& \frac{1}{T}
\sum_{k\in\Z} a_{k}^{(n-sp)/n} T^{k}
=
\sum_{k\in\Z} a_{k+1}^{(n-sp)/n} T^{k}\\&&\qquad=
\sum_{{k\in\Z}\atop{a_k\ne 0}} a_{k+1}^{(n-sp)/n} T^{k}
\\&&\qquad=
\sum_{{k\in\Z}\atop{a_k\ne0}} \Big( a_{k}^{sp/(n\beta)}
T^{k/\alpha}\Big) \Big( a_{k+1}^{1/\beta} a_{k}^{-sp/
(n\beta)} T^{k/\beta}\Big)\\
&&\qquad\le
\left(
\sum_{k\in\Z} \Big( a_{k}^{sp/(n\beta)}
T^{k/\alpha}\Big)^\alpha\right)^{1/\alpha} \left(
\sum_{{k\in\Z}\atop{a_k\ne 0}}\Big( a_{k+1}^{1/\beta}
a_{k}^{-sp/(n\beta)} T^{k/\beta}\Big)^\beta\right)^{1/\beta}
\\
&&\qquad\le
\left(
\sum_{k\in\Z} a_{k}^{(n-sp)/n} T^{k}\right)^{sp/n} \left(
\sum_{{k\in\Z}\atop{a_k\ne0}} a_{k+1}
a_{k}^{-sp/n} T^{k} \right)^{(n-sp)/n}\!.
\end{eqnarray*}
So, recalling \eqref{are c},
we obtain the desired
result.
\end{proof}

We use the above tools to
deal with the measure theoretic properties of the
level sets of the functions (see \cite[Lemma 6]{SV11}).

\begin{lemma}\label{OS2}
Let $s\in (0,1)$ and $p\in[1,+\infty)$ such that $sp<n$. Let 
\begin{equation}\label{XX}{\mbox{
$f\in L^\infty(\R^n)$ be compactly
supported. }}\end{equation}
For any $k\in\Z$ let
\begin{equation}\label{def_ak}
a_k:=
\big| \{ |f|>2^k\} \big|.
\end{equation}
Then,
$$ \int_{\R^n}\int_{\R^n} \frac{|f(x)-f(y)|^p}{|x-y|^{n+sp}}\,dx\,dy\ge
C \sum_{{k\in\Z}\atop{a_k\ne0}}
a_{k+1} a_k^{-sp/n} 2^{pk},$$
for a suitable constant $C=C(n,p,s)>0$.
\end{lemma}

\begin{proof} Notice that
$$ \big| |f(x)|-|f(y)|\big|\le |f(x)-f(y)|,$$
and so,
by possibly replacing~$f$ with~$|f|$, we may consider the case in
which~$f\ge0$.

We define
\begin{equation}\label{t6}
A_k:=\{ |f|>2^k\}.\end{equation}
We remark that~$A_{k+1}\subseteq A_k$, hence
\begin{equation}\label{X0}
a_{k+1}\le a_k .
\end{equation}
We define
$$ D_k:=A_k\setminus A_{k+1}=\{ 2^k<f\le 2^{k+1}\}
\qquad{\mbox{
and }}\qquad d_k:=|D_k|.$$
Notice that
\begin{equation}\label{YY}{\mbox{$d_k$
and $a_k$ are bounded and they become zero when
$k$ is large enough,}}\end{equation} thanks to~\eqref{XX}. Also,
we observe that the $D_k$'s are disjoint, that
\begin{equation}\label{X1}
\bigcup_{{\ell\in\Z}\atop{\ell\le k}} D_\ell\,=\,\CC A_{k+1}
\end{equation}
and that
\begin{equation}\label{X2}
\bigcup_{{\ell\in\Z}\atop{\ell\ge k}} D_\ell\,=\, A_{k}.
\end{equation}
As a consequence of \eqref{X2}, we have that
\begin{equation}\label{YYY}
a_k=\sum_{{\ell\in\Z}\atop{\ell\ge k}} d_\ell
\end{equation}
and so
\begin{equation}\label{X3}
d_k=a_k -\sum_{{\ell\in\Z}\atop{\ell\ge k+1}} d_\ell.
\end{equation}
We stress that the series in \eqref{YYY}
is convergent,
due to~\eqref{YY}, thus so is the series
in~\eqref{X3}. Similarly, we can define
the convergent series
\begin{equation}\label{SS4}
S:=\sum_{{\ell\in\Z}\atop{
a_{\ell-1}\ne 0
}} 2^{p\ell} a_{\ell-1}^{-sp/n}d_\ell.\end{equation}
We notice that~$D_k\subseteq A_k\subseteq A_{k-1}$,
hence~$a_{i-1}^{-sp/n}
d_\ell \le a_{i-1}^{-sp/n} a_{\ell-1}$. Therefore
\begin{equation}\label{FT}
\begin{split}&
\Big\{ (i,\ell)\in \Z {\mbox{ s.t. }}
a_{i-1}\ne 0 {\mbox{ and }}
a_{i-1}^{-sp/n} d_\ell \ne 0\Big\} \\
&\qquad \qquad \qquad \qquad\qquad \qquad\quad
\,\subseteq\,
\Big\{ (i,\ell)\in \Z {\mbox{ s.t. }}
a_{\ell-1}\ne 0 \Big\}.\end{split}
\end{equation}
We use \eqref{FT}
and~\eqref{X0} in the following computation:
\begin{eqnarray}\label{X5}
\sum_{ {i\in\Z}\atop{a_{i-1}\ne0}}
\sum_{ {\ell\in\Z}\atop{\ell\ge i+1} } 2^{pi} a_{i-1}^{-sp/n} d_\ell
\! & = & \!
\sum_{{i\in\Z}\atop{a_{i-1}\ne0}}
\sum_{ {\ell\in\Z}\atop{ {\ell\ge i+1}\atop{
a_{i-1}^{sp/n} d_\ell \ne 0} }} 2^{pi} a_{i-1}^{-sp/n}
d_\ell
\nonumber\\
&\le &\!
\sum_{i\in\Z}
\sum_{ {\ell\in\Z}\atop{ {\ell\ge i+1}\atop{a_{\ell-1}\ne0} } }
2^{pi}  a_{i-1}^{-sp/n} d_\ell
\nonumber\\
& = & \!\sum_{ {\ell\in\Z}\atop{a_{\ell-1}\ne0} } 
\sum_{ {i\in\Z}\atop{i\le\ell-1} } 2^{pi} a_{i-1}^{-sp/n} d_\ell
\nonumber\\
& \le & \!
\sum_{{\ell\in\Z}\atop{a_{\ell-1}\ne0}} \sum_{{i\in\Z}\atop{
i\le\ell-1}}
2^{pi} a_{\ell-1}^{-sp/n} d_\ell
\nonumber\\
& = &\!
\sum_{{\ell\in\Z}\atop{a_{\ell-1}\ne0}} \sum_{k=0}^{+\infty}
2^{p(\ell-1)} 2^{-pk} a_{\ell-1}^{-sp/n} d_\ell
\ \le \  S.
\end{eqnarray}
Now, we fix $i\in\Z$ and~$x\in D_i$: then, for any~$j\in\Z$
with~$j\le i-2$ and any~$y\in D_j$ we have that
$$
|f(x)-f(y)| \, \ge \, 2^i-2^{j+1} \, \ge \,  2^i-2^{i-1} \, = \, 2^{i-1}
$$
and therefore, recalling~\eqref{X1},
\begin{eqnarray*}
\sum_{{j\in\Z}\atop{j\le i-2}}\int_{D_j}
\frac{|f(x)-f(y)|^p}{|x-y|^{n+sp}}\,dy
\! & \ge & \! 2^{p(i-1)}
\sum_{{j\in\Z}\atop{j\le i-2}}\int_{D_j}
\frac{dy}{|x-y|^{n+sp}}\\
& = & \! 2^{p(i-1)} \int_{\CC A_{i-1}}
\frac{dy}{|x-y|^{n+sp}}.
\end{eqnarray*}
This and Lemma~\ref{5yhh}
imply that, for any~$i\in\Z$
and any~$x\in D_i$, we have that
$$ \sum_{ {j\in\Z}\atop{j\le i-2} }\int_{D_j}
\frac{|f(x)-f(y)|^p}{|x-y|^{n+sp}}\,dy
\ge c_o 2^{pi} a_{i-1}^{-sp/n},$$
for a suitable~$c_o>0$.

As a consequence,
for any~$i\in\Z$,
\begin{equation}\label{X4a}
\sum_{ {j\in\Z}\atop{j\le i-2} }\int_{D_i\times D_j}
\frac{|f(x)-f(y)|^p}{|x-y|^{n+sp}}\,dx\,dy
\ge c_o 2^{pi} a_{i-1}^{-sp/n} d_i\,.
\end{equation}
Therefore, by~\eqref{X3}, we conclude that,
for any~$i\in\Z$,
\begin{eqnarray}\label{X4}
\sum_{ {j\in\Z}\atop{j\le i-2} }\int_{D_i\times D_j}
\!\frac{|f(x)-f(y)|^p}{|x-y|^{n+sp}}\,dx\,dy
\ge c_o \left[ 2^{pi} a_{i-1}^{-sp/n} a_i
-\sum_{{\ell\in\Z}\atop{\ell\ge i+1}
}2^{pi} a_{i-1}^{-sp/n} d_\ell 
\right].\nonumber \\
\end{eqnarray}
By \eqref{SS4} and \eqref{X4a}, we have that
\begin{equation}\label{X5bis}
\sum_{{i\in\Z}\atop{a_{i-1}\ne0}}
\sum_{{j\in\Z}\atop{j\le i-2}}\int_{D_i\times D_j}
\frac{|f(x)-f(y)|^p}{|x-y|^{n+sp}}\,dx\,dy
\ge c_o S.\end{equation}
Then, using~\eqref{X4}, \eqref{X5} and~\eqref{X5bis},
\begin{eqnarray*}
&& \sum_{{i\in\Z}\atop{a_{i-1}\ne0}}
\sum_{{j\in\Z}\atop{j\le
i-2}}\int_{D_i\times D_j}
\frac{|f(x)-f(y)|^p}{|x-y|^{n+sp}}\,dx\,dy
\\
&& \qquad \qquad  \geq \,
c_o \!\left[ \sum_{{i\in\Z}\atop{a_{i-1}\ne0}}
2^{pi} a_{i-1}^{-sp/n} a_i
\,-\,
\sum_{{i\in\Z}\atop{a_{i-1}\ne0}}
\sum_{{\ell\in\Z}
\atop{\ell\ge i+1}}2^{pi} a_{i-1}^{-sp/n} d_\ell\right]
\\
&& \qquad \qquad \geq \, c_o\!\!\left[ \sum_{{i\in\Z}\atop{a_{i-1}\ne0}}
2^{pi} a_{i-1}^{-sp/n} a_i
\,-\, S
\right]
\\
&& \qquad \qquad  \geq \,  c_o
\!\sum_{{i\in\Z}\atop{a_{i-1}\ne0}}
2^{pi} a_{i-1}^{-sp/n} a_i
- \!\!\sum_{{i\in\Z}\atop{a_{i-1}\ne0}}
\sum_{{j\in\Z}\atop{j\le
i-2}}\!\int_{D_i\times D_j}
\!\frac{|f(x)-f(y)|^p}{|x-y|^{n+sp}}\,dx\,dy.
\end{eqnarray*}
That is, by taking the last term to
the left hand side,
\begin{equation}\label{XF}
\sum_{{i\in\Z}\atop{a_{i-1}\ne0}}
\sum_{{j\in\Z}\atop{j\le
i-2}}\!\!\int_{D_i\times D_j}
\frac{|f(x)-f(y)|^p}{|x-y|^{n+sp}}\,dx\,dy
\,\ge\,
c_o\!\!\sum_{{i\in\Z}\atop{a_{i-1}\ne0}}
2^{pi} a_{i-1}^{-sp/n} a_i,
\end{equation}
up to relabeling the constant $c_0$.
\vspace{1mm}

On the other hand, by symmetry,
\begin{eqnarray}\label{XFF}
&& \int_{\R^n\times\R^n}
\frac{|f(x)-f(y)|^p}{|x-y|^{n+sp}}\,dx\,dy \nonumber \\
&& \qquad \qquad \quad = \, \sum_{{i,j\in\Z}} \int_{D_i\times D_j}
\frac{|f(x)-f(y)|^p}{|x-y|^{n+sp}}\,dx\,dy \nonumber\\
\nonumber \\
&& \qquad \qquad \quad \geq 
2\sum_{{i,j\in\Z}\atop{j< i}}
\int_{D_i\times D_j}
\frac{|f(x)-f(y)|^p}{|x-y|^{n+sp}}\,dx\,dy \nonumber \\
\nonumber \\
&& \qquad \qquad \quad  \ge  
2\sum_{{i\in\Z}\atop{a_{i-1}\ne0}}
\sum_{{j\in\Z}\atop{j\le
i-2}}\int_{D_i\times D_j}
\frac{|f(x)-f(y)|^p}{|x-y|^{n+sp}}\,dx\,dy.
\end{eqnarray}
Then, the desired result plainly follows
from~\eqref{XF} and~\eqref{XFF}.
\end{proof}
\vspace{1mm}

\begin{lemma}\label{lem_troncata} Let $q\in [1,\infty)$. Let $f:\R^n\to\R$ be a measurable function. For any $N\in \N$, let
\begin{equation}\label{def_troncata}
f_N(x):=\max\big\{\min\{f(x),N\},\, -N\big\} \ \ \forall x \in\R^n.
\end{equation}
Then
\begin{equation*}
\lim_{N\to+\infty}\|f_N\|_{L^q(\R^n)}=\|f\|_{L^q(\R^n)}.
\end{equation*}
\end{lemma}
\begin{proof}
We denote by $|f|_N$ the function obtained by cutting $|f|$ at level $N$. We have that $|f|_N=|f_N|$ and so, by Fatou Lemma, we obtain that
\begin{eqnarray*}
\liminf_{N\to+\infty}\|f_N\|_{L^q(\R^n)} \! & = & \! \liminf_{N\to+\infty}\left(\int_{\R^n}|f|^q_N\right)^{\!\frac{1}{q}}
 \, \geq \,  \left(\int_{R^n}|f|^q\right)^{\!\frac{1}{q}} \, =  \, \|f\|_{L^q(\R^n)}.
\end{eqnarray*}
The reverse inequality easily follows by the fact that $|f|_N(x)\leq |f(x)|$ for any $x\in \R^n$.
\end{proof}

\vspace{1mm}

Taking into account the previous lemmas, we are able to give an elementary proof of the Sobolev-type inequality stated in the following theorem.

\begin{thm}\label{thm_sobolev}
Let $s\in (0,1)$ and $p\in[1,+\infty)$ such that $sp<n$.
Then there exists a positive constant $C=C(n,p,s)$ such that, for any measurable and compactly
supported function $f:\R^n\rightarrow\R$, we have
\begin{equation}\label{eq_sobolev} \| f\|_{L^{p^{\star}} (\R^n)}^p\le C
\int_{\R^n} \int_{\R^n} \frac{|f(x)-f(y)|^p}{|x-y|^{n+sp}}
\,dx\,dy.\end{equation}
where ${p^{\star}}=p^{\star}(n,s)$ is the so-called ``fractional critical exponent'' and it is equal to $np/(n-sp)$.\vspace{0.5mm}

Consequently, the space $W^{s,p}(\R^n)$ is continuously embedded in $L^q(\R^n)$ for any $q\in [p, p^\star]$.
\end{thm}

\begin{proof}
First, we note that if the right hand side of~\eqref{eq_sobolev} is unbounded then the claim in the theorem plainly follows. Thus, we may suppose that $f$ is such that
\begin{equation}\label{eq_semi}
\int_{\R^n}\int_{\R^n}\frac{|f(x)-f(y)|^p}{|x-y|^{n+sp}}\,dx\,dy \, < \, +\infty.
\end{equation}
Moreover, we can suppose, without loss of generality, that
\begin{equation}\label{vega}
f \in L^{\infty}(\R^n).
\end{equation}
Indeed, if~\eqref{eq_semi} holds for bounded functions, then it holds also for the function $f_N$, obtained by any (possibly unbounded) $f$ by cutting at levels $-N$ and $+N$ (see~\eqref{def_troncata}). Therefore, by~Lemma~\ref{lem_troncata} and the fact that~\eqref{eq_semi} together with the Dominated Convergence Theorem imply
$$
\lim_{N\to+\infty}\int_{\R^n}\int_{\R^n}\frac{|f_N(x)-f_N(y)|^p}{|x-y|^{n+sp}}\,dx\,dy \, = \, \int_{\R^n}\int_{\R^n}\frac{|f(x)-f(y)|^p}{|x-y|^{n+sp}}\,dx\,dy,
$$
we obtain estimate \eqref{eq_sobolev} for the function $f$.
\vspace{2mm}

Now, take $a_k$ and $A_k$ defined by~\eqref{def_ak} and~\eqref{t6}, respectively. We have
\begin{eqnarray*}
\| f\|_{L^{p^{\star}} (\R^n)}^{p^{\star}}
\! & = & \! \sum_{k\in\Z}
\int_{A_k\setminus A_{k+1}}\! |f(x)|^{p^{\star}}\,dx
\ \leq \ \sum_{k\in\Z}
\int_{A_k\setminus A_{k+1}} (2^{k+1})^{p^{\star}} \,dx
\\
\\
& \le & \! \sum_{k\in\Z} 2^{(k+1){p^{\star}}} a_k.
\end{eqnarray*}
That is,
$$
\| f\|_{L^{p^{\star}} (\R^n)}^p
\, \le \, 2^p \left( \sum_{k\in\Z} 2^{k{p^{\star}}} a_k\right)^{p/{p^{\star}}}\!\!.
$$
Thus, since $p/{p^{\star}}\, =\, {(n-sp)/n} \, = \, 1-{sp}/{n} \, < \,1$,
\begin{equation}\label{eq_4star}
\dys \| f\|_{L^{p^{\star}} (\R^n)}^p
\, \le \,  2^p \sum_{k\in\Z} 2^{kp}
a_k^{(n-sp)/n}
\end{equation}
and, then, by choosing $T=2^p$, Lemma~\ref{OS1} yields
\begin{equation}\label{eq_5star}
\dys \|f\|^p_{L^{p^{\star}}(\R^n)} \, \le \, C\sum_{{k\in \Z}\atop{a_k\neq0}} 2^{kp}a_{k+1}a_k^{-\frac{sp}{n}},
\end{equation}
for a suitable constant $C$ depending on $n, p$ and $s$.

Finally, it suffices to apply~Lemma~\ref{OS2} and we obtain the desired result, up to relabeling the constant $C$ in~\eqref{eq_5star}.\\
Furthermore, the embedding for $q\in(p,p^{\star})$ follows from standard application of H\"{o}lder inequality.
\end{proof}

\begin{rem}
From Lemma~\ref{5yhh}, it follows that
\begin{equation}\label{CA3bis}\int_E \int_{\CC
E}\frac{dx\,dy}{|x-y|^{n+sp}}\ge
c(n,s)\,|E|^{(n-sp)/n}\end{equation}
for all
measurable sets~$E$ with finite measure.

On the other hand, we see that~\eqref{eq_sobolev} reduces to~\eqref{CA3bis} 
when~$f=\chi_E$,
so~\eqref{CA3bis} (and thus
Lemma~\ref{5yhh}) may be seen as a Sobolev-type inequality for sets.
\end{rem}

\vspace{2mm}

The above embedding does not generally hold for the space $W^{s,p}(\Omega)$ since it not always possible to extend a function $f\in W^{s,p}(\Omega)$ to a function $\tilde{f}\in W^{s,p}(\R^n)$. In order to be allowed to do that, we should require further regularity assumptions on $\Omega$ (see Section \ref{sec_estensione}).

\begin{thm}\label{thm_sobolev1}
Let $s\in (0,1)$ and $p\in[1,+\infty)$ such that $sp<n$. Let $\Omega \subseteq \R^n$ be an extension domain for $W^{s,p}$.
Then there exists a positive constant $C=C(n,p,s,\Omega)$ 
such that, for any $f\in W^{s,p}(\Omega)$, we have
\begin{equation}\label{eq_sobolev1}
\| f\|_{L^{q}(\Omega)}\le C
\|f\|_{W^{s,p}(\Omega)},
\end{equation}
for any $q \in [p,p^{\star}]$; i.e., the space $W^{s,p}(\Omega)$ is continuously embedded in~$L^q(\Omega)$ for any $q\in [p, p^\star]$.\vspace{1mm}

If, in addition, $\Omega$ is bounded, then the space~$W^{s,p}(\Omega)$ is continuously embedded in~$L^q(\Omega)$ for any $q\in [1, p^\star]$.
\end{thm}

\begin{proof}
Let $f\in W^{s,p}(\Omega)$. Since $\Omega \subseteq \R^n$ is an extension domain for $W^{s,p}$, then there exists a constant $C_1=C_1(n,p,s,\Omega)>0$ such that
\begin{equation}\label{eq_distro1}
\|\tilde{f}\|_{W^{s,p}(\R^n)} \leq C_1 \|f\|_{W^{s,p}(\Omega)},
\end{equation}
with $\tilde{f}$ such that $\tilde{f}(x)=f(x)$ for $x$ a.e. in $\Omega$.

On the other hand, by Theorem~\ref{thm_sobolev}, the space~$W^{s,p}(\R^n)$ is continuously embedded in~$L^{q}(\R^n)$ for any $q\in[p,p^{\star}]$; i.e., there exists a constant $C_2=C_2(n,p,s)>0$ such that
\begin{equation}\label{eq_distro2}
\|\tilde{f}\|_{L^q(\R^n)} \leq C_2 \|\tilde{f}\|_{W^{s,p}(\R^n)}.
\end{equation}
Combining~\eqref{eq_distro1} with~\eqref{eq_distro2}, we get
\begin{eqnarray*}
\|f\|_{L^q(\Omega)} \!\! & =  & \!\! \|\tilde{f}\|_{L^q(\Omega)} \, \leq  \, \|\tilde{f}\|_{L^q(\R^n)} \, \leq \, C_2\|\tilde{f}\|_{W^{s,p}(\R^n)} \\
&  \leq  & \!\! C_2 C_1 \|f\|_{W^{s,p}(\Omega)},
\end{eqnarray*}
that gives the inequality in~\eqref{eq_sobolev1}, by choosing~$C=C_2 C_1$.

\vspace{1mm}

In the case of $\Omega$ being  bounded, the embedding for $q\in[1,p)$ plainly follows from~\eqref{eq_sobolev1}, by using the H\"{o}lder inequality.
\end{proof}

\vspace{1mm}

\begin{rem} In the critical case
$q=p^\star$ the constant $C$ in~Theorem~\ref{thm_sobolev1} does not depend
on~$\Omega$: this is a consequence of~\eqref{eq_sobolev}
and of the extension property of~$\Omega$. \end{rem}

\vspace{3mm}

\subsection{The case $sp=n$}\label{sec_spn}
We note that when $sp\to n$ the critical exponent $p^\star$ goes to $\infty$ and so it is not surprising that, in this case, if $f$ is in $W^{s,p}$ then $f$ belongs to $L^q$ for any $q$, as stated in the following two theorems.
\begin{thm}\label{thm_spn}
Let $s\in (0,1)$ and $p\in[1,+\infty)$ such that $sp=n$.
Then there exists a positive constant $C=C(n,p,s)$ such that, for any measurable and compactly
supported function $f:\R^n\rightarrow\R$, we have
\begin{equation}\label{eq_spn} \| f\|_{L^{q} (\R^n)}\le C
\|f\|_{W^{s,p}(\R^n)},
\end{equation}
for any $q\in [p,\infty)$; i.e.,
the space $W^{s,p}(\R^n)$ is continuously embedded in $L^q(\R^n)$ for any $q\in [p, \infty)$.
\end{thm}
\vspace{2mm}

\begin{thm}\label{thm_spn1}
Let $s\in (0,1)$ and $p\in[1,+\infty)$ such that $sp=n$. Let $\Omega \subseteq \R^n$ be an extension domain for $W^{s,p}$.
Then there exists a positive constant $C=C(n,p,s,\Omega)$ such that, for any $f\in W^{s,p}(\Omega)$, we have
\begin{equation}\label{eq_spn1} \| f\|_{L^{q} (\Omega)}\le C
\|f\|_{W^{s,p}(\Omega)},
\end{equation}
for any $q\in [p,\infty)$; i.e., the space $W^{s,p}(\Omega)$ is continuously embedded in~$L^q(\Omega)$ for any $q\in[p,\infty)$.\vspace{1mm}

If, in addition, $\Omega$ is bounded, then the space~$W^{s,p}(\Omega)$ is continuously embedded in~$L^q(\Omega)$ for any $q\in [1, \infty)$.
\end{thm}

The proofs can be obtained by simply combining Proposition~\ref{enrico} with Theorem~\ref{thm_sobolev} and Theorem~\ref{thm_sobolev1}, respectively.

\vspace{3mm}

\section{Compact embeddings}\label{sec_compactness}
In this section, we state and prove some compactness results involving the fractional spaces~$W^{s,p}(\Omega)$ in bounded domains. The main proof is a modification of the one of the classical Riesz-Frechet-Kolmogorov Theorem (see~\cite{Kol31, Rie33}) and, again, it is self-contained and it does not require to use Besov  or other interpolation spaces, nor Fourier transform and semigroup flows (see \cite[Theorem 1.5]{CT04b}). We refer to~\cite[Lemma 6.11]{PSV11} for the case $p=q=2$.
\vspace{1mm}

\begin{thm}\label{thm_comp}
Let $s\in (0,1)$, $p\in [1,+\infty)$, $q\in[1, p]$, $\Omega\subset\R^n$
be a bounded extension domain for $W^{s,p}$
and $\TT$ be a bounded subset of 
$L^p(\Omega)$.
Suppose that
$$ \sup_{f\in \TT} 
\int_{\Omega}\int_{\Omega}
\frac{|f(x)-f(y)|^p}{|x-y|^{n+sp}}\,dx\,dy
\,<\,+\infty.$$
Then $\TT$ is pre-compact in $L^q(\Omega)$.
\end{thm}

\begin{proof} 
We want to show that $\TT$ is totally bounded in $L^q(\Omega)$,
i.e., for any $\eps \in (0,1)$ there exist $\beta_1,\dots,\beta_M \in L^q(\Omega)$
such that for any $f\in \TT$ there exists $j\in \{ 1,\dots, M\}$ such that
\begin{equation}\label{asd}
\| f-\beta_j\|_{L^q(\Omega)} \le \eps.
\end{equation}

Since $\Omega$ is an extension domain, there exists a function $\tilde{f}$ in $W^{s,p}(\R^n)$ such that $\|\tilde{f}\|_{W^{s,p}(\R^n)}\leq C \|f\|_{W^{s,p}(\Omega)}$. 
 Thus, for any cube $Q$ containing $\Omega$, we have
$$\|\tilde{f}\|_{W^{s,p}(Q)}\, \leq\,  \|\tilde{f}\|_{W^{s,p}(\R^n)} \, \leq \,  C\|f\|_{W^{s,p}(\Omega)}.$$

Observe that, since $Q$ is a bounded open set, $\tilde{f}$ belongs also to $L^q(Q)$ for any $q\in[1,p]$.
Now, for any $\eps\in (0,1)$,  we let 
\begin{eqnarray*}&& C_o:= 1+\sup_{f\in \TT}\|\tilde{f}\|_{L^q(Q)}+\sup_{f\in\TT}
\int_{Q} \int_{Q}
\frac{|\tilde{f}(x)-\tilde{f}(y)|^p}{|x-y|^{n+sp}}\,dx\,dy,\\
&&\rho=\rho_{\eps}:= \left(
\frac{\eps}{2C_o^{\frac{1}{q}}\,n^{\frac{n+sp}{2p}}}
\right)^{\!\frac{1}{s}}
\ {\mbox{
and
}}\ \ \
\eta=\eta_{\eps}\!:=\frac{\eps\,\rho^{\frac{n}{q}}}2,\end{eqnarray*}
and we take a collection of disjoints 
cubes $Q_1,\dots,Q_N$
of side 
$\rho$
such that\footnote{To be precise,
for this one needs to take $\varepsilon\in(0,1)$
arbitrarily small and such that the ration between
the side of~$Q$ and~$\rho_\varepsilon$ is integer.}
$$
\Omega\subseteq Q	=\bigcup_{j=1}^N Q_j.
$$
For any $x\in \Omega$, 
we define
\begin{equation}\label{def_jx}
j(x) \ \text{as the unique integer in}  \ \{ 1,\dots,N\} \ \text{for which}
 \ x\in Q_{j(x)}.
\end{equation}
Also, for any $f\in\TT$, let
$$ P(f) (x):=
\frac{1}{|Q_{j(x)}|} 
\int_{Q_{j(x)}}
\tilde{f}(y)\,dy.$$
Notice that
$$
P(f+g)=P(f)+P(g) \ \ \ \text{for any} \ f, g \in \TT
$$
and that $P(f)$ is constant, say equal to $q_j(f)$, in any $Q_j$,
for $j\in\{ 1,\dots,N\}$. Therefore, we can define
$$ R(f) :=\rho^{n/q} \big( q_1(f),\dots,q_N(f)\big)\in \R^N$$
and consider the spatial $q$-norm in $\R^N$ as
$$
\|v\|_q:=\left( \sum_{j=1}^N|v_j|^q\right)^{\!\frac{1}{q}}\!\!, \ \ \text{for any} \ v \in \R^N.
$$
We observe that $R({f+g})=R(f)+R(g)$. Moreover, 
\begin{eqnarray}\label{509}
\|P(f)\|^q_{L^q(\Omega)} \! &  = &\! \sum_{j=1}^N\int_{Q_j\cap\Omega}|P(f)(x)|^q\, dx \nonumber \\
\nonumber \\
& \leq & \! \rho^n\sum_{j=1}^N |q_j(f)|^q \ = \ \|R(f)\|_q^q \ \leq \ \frac{\|R(f)\|_q^q}{\rho^n}.
\end{eqnarray}
Also, by H\"older
inequality,
\begin{eqnarray*}
\|R(f)\|_q^q \! & = & \! \sum_{j=1}^N\rho^n|q_j(f)|^q
\, = \, \frac{1}{\rho^{n(q-1)}}
\sum_{j=1}^N \left|
\int_{Q_{j}} \tilde{f}(y)\,dy
\right|^q  \\
& \leq & \!
\sum_{j=1}^N 
\int_{Q_{j}} |\tilde{f}(y)|^q\,dy
\, = \,
\int_{Q} |\tilde{f}(y)|^q \, dy \, =\, \|\tilde{f}\|_{L^q(Q)}^q.
\end{eqnarray*}
In particular,
$$
\sup_{f\in\TT}\|R(f)\|_q^q
\le C_o,
$$
that is, the set $R(\TT)$ is bounded in $\R^N$ (with respect to the $q$-norm of $\R^N$ as well as to any equivalent norm of $\R^N$)
and so, since it is finite dimensional, it is totally bounded.
Therefore, there exist $b_1,\dots,b_M\in \R^N$ such that
\begin{equation}\label{eq_22bis}
R(\TT)\subseteq \bigcup_{i=1}^M B_\eta (b_i),
\end{equation}
where the balls $B_\eta$ are taken in the $q$-norm of $\R^N$.
\vspace{1mm}

For any $i\in \{1,\dots, M\}$, we write the coordinates of $b_i$ as $$b_i=(b_{i,1},\dots, b_{i,N})\in \R^N.$$
For any $x\in \Omega$, we set
$$
\beta_i(x):= \rho^{-n/q}\,b_{i,j(x)},
$$
where $j(x)$ is as in~\eqref{def_jx}.

Notice that $\beta_i$ is constant on $Q_j$, i.e. if $x\in Q_j$ then 
\begin{equation}\label{eq_2t2}
P(\beta_i)(x)=\rho^{-\frac{n}{q}}b_{i,j}=\beta_{i}(x)
\end{equation}
and so $q_j(\beta_i)=\rho^{-\frac{n}{q}}b_{i,j}$; thus 
\begin{equation}\label{eq_25bis}
 R(\beta_i)=b_i.
 \end{equation}

Furthermore, for any $f\in \TT$
\begin{eqnarray}\label{comp}
\| f-P(f)\|_{L^q(\Omega)}^q \! & = & \!\sum_{j=1}^N
\int_{Q_j\cap \Omega} |f(x)-P(f)(x)|^q\,dx \nonumber
\\
\nonumber
\\ & = & \!  \sum_{j=1}^N\int_{Q_j\cap \Omega} \left| f(x)-\frac{1}{|Q_j|} \int_{Q_j} 
\tilde{f}(y)\,dy\right|^q\!dx \nonumber
\\
\nonumber
\\ & = & \! \sum_{j=1}^N
\int_{Q_j\cap \Omega} \frac{1}{|Q_j|^q}\left|\int_{Q_j} f(x)-
\tilde{f}(y)\,dy\right|^q\!dx \nonumber
\\ & \le & \! \frac{1}{\rho^{nq}}\,\sum_{j=1}^N
\int_{Q_j\cap \Omega} \left[\int_{Q_j} |f(x)-
\tilde{f}(y)|\,dy\right]^q\!dx.
\end{eqnarray}

Now for any fixed $j\in{1,\cdots,N}$, by H\"{o}lder inequality with $p$ and $p/(p-1)$ we get
\begin{eqnarray}\label{comp1}
&& \frac{1}{\rho^{nq}}\, \left[\int_{Q_j} |f(x)-\tilde{f}(y)|\,dy\right]^q \nonumber \\
&&  \qquad \qquad \qquad \qquad \, \le \,\frac{1}{\rho^{nq}} \, |Q_j|^{\frac{q(p-1)}{p}} \left[\int_{Q_j} \big|f(x)-
\tilde{f}(y)\big|^p\,dy \right]^{\frac{q}{p}}  \nonumber
\\
\nonumber
\\
&&  \qquad \qquad \qquad \qquad = \, \frac{1}{\rho^{nq/p}} \left[\int_{Q_j} \big|f(x)-
\tilde{f}(y)\big|^p\,dy \right]^{\frac{q}{p}} \nonumber
\\ 
\nonumber
\\
&&  \qquad \qquad \qquad \qquad \le \, \frac{1}{\rho^{nq/p}}\, n^{\left(\frac{n+sp}{2}\right)\frac{q}{p}}\rho^{\frac{q}{p}(n+sp)}\left[\int_{Q_j} \frac{|f(x)-
\tilde{f}(y)|^p}{|x-y|^{n+sp}}\,dy\right]^{\frac{q}{p}}  \nonumber
\\
\nonumber
\\
&&  \qquad \qquad \qquad \qquad \le \, n^{\left(\frac{n+sp}{2}\right)\frac{q}{p}}\rho^{sq} \left[\int_{Q} \frac{|f(x)-
\tilde{f}(y)|^p}{|x-y|^{n+sp}}\,dy\right]^{\frac{q}{p}}.
\end{eqnarray}
Hence, combining \eqref{comp} with \eqref{comp1}, we obtain that
\begin{eqnarray}\label{comp2}
\|f-P(f)\|^q_{L^q(\Omega)} & \le & \! n^{\left(\frac{n+sp}{2}\right)\frac{q}{p}}\rho^{sq}
\int_{Q} \left[\int_{Q} \frac{|\tilde{f}(x)-
\tilde{f}(y)|^p}{|x-y|^{n+sp}}\,dy\right]^{\frac{q}{p}}\!dx  \nonumber
\\
\nonumber
\\
& \le & \! n^{\left(\frac{n+sp}{2}\right)\frac{q}{p}}\rho^{sq}
\left[ \int_{Q} \int_{Q} \frac{|\tilde{f}(x)-
\tilde{f}(y)|^p}{|x-y|^{n+sp}}\,dy\, dx \,\right]^{\frac{q}{p}}
\\
\nonumber
\\ & \le & \! C_o\;n^{\left(\frac{n+sp}{2}\right)\frac{q}{p}}\,\rho^{sq} \, = \, \frac{\eps^q}{2^q} \nonumber. 
\end{eqnarray}
where \eqref{comp2} follows from Jensen inequality since $t\mapsto |t|^{q/p}$ is a concave function for any fixed $p$ and $q$ such that $q/p \leq 1$.

\vspace{1mm}

Consequently, for any $j\in\{1, ..., M\}$, recalling~\eqref{509} and~\eqref{eq_2t2}
\begin{eqnarray}\label{905}
\|f-\beta_j\|_{L^q(\Omega)} \! & \le & \!
\|f-P(f)\|_{L^q(\Omega)}+\|P(\beta_j)-\beta_j\|_{L^q(\Omega)}
+\|P(f-\beta_j)\|_{L^q(\Omega)} \nonumber 
\\
& \leq & \!  \frac{\eps}{2}+\frac{\|R(f)-R(\beta_j)\|_q}{\rho^{n/q}}.
\end{eqnarray}
\vspace{1mm}

Now, given any~$f\in\TT$, we recall~\eqref{eq_22bis} and~\eqref{eq_25bis} and we take~$j\in\{1,\dots,M\}$ such that~$R(f)\in
B_\eta(b_j)$.
Then,~\eqref{eq_2t2} and~\eqref{905} give that
\begin{eqnarray}\label{ele2}
\| f-\beta_j\|_{L^q(\Omega)} \ \le
\ \frac{\eps}{2}+\frac{\|R(f)-b_j\|_q}{\rho^{n/q}}\ \le \
\frac\eps2+\frac{\eta}{\rho^{n/q}}\ = \ \eps.
\end{eqnarray}
This proves \eqref{asd}, as desired. 
\end{proof}

\vspace{1mm}

\begin{corollary}\label{cor1}
Let  $s\in (0,1)$ and $p\in [1,+\infty)$ such that $sp<n$. Let $q\in[1, p^{\star})$, $\Omega\subseteq\R^n$ be a bounded extension domain for~$W^{s,p}$
and $\TT$ be a bounded subset~of~
$L^p(\Omega)$.
Suppose that
$$ \sup_{f\in \TT} 
\int_{\Omega}\int_{\Omega}
\frac{|f(x)-f(y)|^p}{|x-y|^{n+sp}}\,dx\,dy
\,<\,+\infty.$$
Then $\TT$ is pre-compact in $L^q(\Omega)$. 
\end{corollary}

\begin{proof}
First, note that for $1 \leq q \leq p$ the compactness follows from~Theorem~\ref{thm_comp}.

For any $q\in (p, p^{\star})$, we may take $\theta=\theta(p,p^{\star},q) \in (0,1)$ such that ${1}/{q}={\theta}/{p}+{1-\theta}/{p^{\star}}$, thus for any $f\in \TT$ and $\beta_j$ with $j\in\{1,...,N\}$ as in the theorem above, using H\"{o}lder inequality with $p/(\theta q)$ and $p^{\star}/((1-\theta)q)$, we get
\begin{eqnarray*}
\|f-\beta_j\|_{L^q(\Omega)} &=& \left(\int_{\Omega} |f-\beta_j|^{q\theta}\, |f-\beta_j|^{q(1-\theta)} \, dx  \right)^{1/q} \\
\\
& \leq & \left(  \int_{\Omega} |f-\beta_j|^{p} \, dx  \right)^{\theta/p} \,  \left(  \int_{\Omega} |f-\beta_j|^{p^{\star}} \, dx  \right)^{(1-\theta)/p^{\star}} \\
\\
&=& \|f-\beta_j\|^{1-\theta}_{L^{p{\star}}(\Omega)}\, \|f-\beta_j\|^{\theta}_{L^p(\Omega)} \,\\
\\
& \leq & C \|f-\beta_j\|^{1-\theta}_{W^{s,p}(\Omega)} \, \|f-\beta_j\|^{\theta}_{L^p(\Omega)} \,\leq\, \tilde{C} \eps^{\theta},
\end{eqnarray*}
where the last inequalities comes directly from~\eqref{ele2} and the continuos embedding (see Theorem~\ref{thm_sobolev1}).
\end{proof}
\begin{rem}
As well known in the classical case $s=1$ (and, more generally, when $s$ is an integer), also in the fractional case the lack of compactness for the critical embedding ($q=p^{\star}$) is not surprising, because of translation and dilation invariance (see~\cite{PP10} for various results in this direction, for any $0<s<n/2$). 
\end{rem}

\vspace{2mm}

Notice that the regularity assumption on~$\Omega$ in Theorem~\ref{thm_comp} and Corollary~\ref{cor1} cannot be dropped (see Example~\ref{exe_ovidiu} in Section~\ref{sec_example}).

\vspace{3mm}

\section{H\"older regularity}\label{sec_holder}

In this section we will show certain regularity properties for functions in $W^{s,p}(\Omega)$ when $sp>n$ and $\Omega$ is an extension domain for $W^{s,p}$ with no external cusps. For instance, one may take $\Omega$ any Lipschitz domain (recall Theorem~\ref{thm_estensione}).

\vspace{1mm}
The main result is stated in the forthcoming Theorem~\ref{thm_holder}. First, we need a simple  technical lemma, whose proof can be found in \cite{Giu94} (for instance).

\vspace{2mm}
\begin{lemma}\label{giusti}{\rm (\cite[Lemma 2.2]{Giu94})}.
Let $p\in [1, +\infty) $ and $sp\in(n,n+p]$. Let $\Omega\subset\R^n$ be a domain with no external cusps and $f$ be a function in $W^{s,p}(\Omega)$. Then, for any $x_0 \in \Omega$ and $R,R'$, with $0< R'< R< \rm{diam}(\Omega)$, we have
\begin{equation}\label{giustieq}
|\langle f \rangle_{B_R(x_0)\cap \Omega}- \langle f\rangle_{B_{R'}(x_0)\cap \Omega}|\leq \,c\, [f]_{p,sp} \,|B_R(x_0)\cap\Omega|^{(sp-n)/np}
\end{equation}
where 
$$
[f]_{p,sp} := \left(\sup_{x_0\in\Omega \, \rho>0}  \, \rho^{-sp} \int_{B_\rho(x_0)\cap \Omega} |f(x)-\langle f\rangle_{B_\rho(x_0)\cap \Omega}|^p \,dx\right)^{\!\frac{1}{p}}
$$
and
$$
\langle f\rangle_{B_\rho(x_0)\cap \Omega}: = \frac{1}{|B_\rho(x_0)\cap \Omega|} \int_{B_\rho(x_0)\cap \Omega} f(x) dx.
$$
\end{lemma}

\vspace{2mm}

\begin{thm}\label{thm_holder}
Let~$\Omega\subseteq\R^n$ be an extension domain for~$W^{s,p}$ with no external cusps and
let~$p\in [1,+\infty)$, $s\in(0,1)$ such that $sp>n$. Then, there exists~$C>0$, depending
on~$n$, $s$, $p$ and~$\Omega$, such that
\begin{equation}\label{001}
\| f\|_{C^{0,\alpha}(\Omega)}
\le C\left(
\| f\|^p_{L^p(\Omega)}+
 \int_\Omega
\int_\Omega
\frac{|f(x)-f(y)|^p}{|x-y|^{n+sp}}\,dx\,dy\right)^{\!\frac{1}{p}}\!,
\end{equation}
for any~$f\in L^p(\Omega)$, with~$\alpha:=(sp-n)/p$.
\end{thm}

\begin{proof}
In the following, we will denote by $C$ suitable positive quantities, possibly different from line to line, and possibly depending on $p$ and $s$.
\vspace{2mm}

First, we notice that if the right hand side of~\eqref{001}
is not finite, then we are done. Thus, we may suppose that
$$ \int_\Omega
\int_\Omega
\frac{|f(x)-f(y)|^p}{|x-y|^{n+sp}}\,dx\,dy\le C,$$
for some $C>0$.

\vspace{2mm}

Second, since $\Omega$ is an extension domain for~$W^{s,p}$, we can extend any $f$ to a function $\tilde{f}$ such that $\|\tilde{f}\|_{W^{s,p}(\R^n)}\leq C\|f\|_{W^{s,p}(\Omega)}$.

\vspace{2mm}

Now, for any bounded measurable set~$U\subset\R^n$, we consider the average value of the function $\tilde{f}$ in $U$, given by
$$
\langle \tilde{f}\rangle_{U}:=\frac{1}{|U|}\!\int_{U} \tilde{f}(x)\,dx.
$$
For any~$\xi\in\R^n$, the H\"older inequality yields
$$
\big| \xi-\langle \tilde{f}\rangle_{U}\big|^p
\, =\, \frac{1}{|U|^p}\left|
\int_U \xi-\tilde{f}(y)\,dy
\right|^p
\, \le \
\frac{1}{|U|} \int_U |\xi-\tilde{f}(y)|^p\,dy.
$$
Accordingly, by taking~ $x_o\in \Omega$ and $U:=B_r(x_o)$,
$\xi:=\tilde{f}(x)$ and integrating over~$B_r(x_o)$, we obtain that
\begin{eqnarray*} && \int_{B_r(x_o)} |\tilde{f}(x)-\langle 
\tilde{f}\rangle_{B_r(x_o)}
|^p\,dx \\&&\qquad\qquad \qquad
\le\frac{1}{|B_r(x_o)|}\int_{B_r(x_o)}\int_{B_r(x_o)}
|\tilde{f}(x)-\tilde{f}(y)|^p\,dx\,dy.\end{eqnarray*}
Hence, since~$|x-y|\le 2r$ for any~$x$, $y\in B_r(x_o)$,
we deduce that
\begin{eqnarray}\label{campana}
&&\int_{B_r(x_o)} |\tilde{f}(x)-\langle 
\tilde{f}\rangle_{B_r(x_o)}|^p\,dx
\nonumber\\ 
\nonumber \\
&&\qquad \qquad  \qquad  \le \, \frac{(2r)^{n+sp}}{|B_r(x_o)|}\int_{B_r(x_o)}\int_{B_r(x_o)}
\frac{|\tilde{f}(x)-\tilde{f}(y)|^p}{|x-y|^{n+sp
}}\,dx\,dy \nonumber \\
\nonumber \\
&& \qquad \qquad\qquad  \le \, \frac{2^{n+sp}\, r^{sp} C \|f\|^p_{W^{s,p}(\Omega)} }{
|B_1|},
\end{eqnarray}
that implies
\begin{equation}\label{campan1}
[f]^p_{p,sp} \, \leq \,  C \|f\|^p_{W^{s,p}(\Omega)},
\end{equation}
for a suitable constant $C$.

\vspace{2mm}

Now, we will show that $f$ is a continuos function. Taking into account \eqref{giustieq}, it follows that the sequence of functions $x\rightarrow \langle f\rangle_{B_R(x)\cap \Omega}$ converges uniformly in $x\in\Omega$ when $R\rightarrow 0$. In particular the limit function $g$ will be continuos and the same holds for $f$, since by Lebesgue theorem we have that
$$
\lim_{R\rightarrow 0} \,\frac{1}{|B_R(x)\cap \Omega|} \,\int_{B_R(x)\cap\Omega} f(y) \,dy=f(x)   \quad\mbox{for almost every}\,\; x\in \Omega.
$$
\vspace{1mm}

Now, take any $x,y\in\Omega$ and set $R=|x-y|$. We have
\begin{equation*}
|f(x)-f(y)|\leq |f(x)-\langle \tilde{f}\rangle_{B_{2R}(x)}|+|\langle \tilde{f} \rangle_{B_{2R}(x)}-\langle \tilde{f} \rangle_{B_{2R}(y)}|+ |\langle \tilde{f} \rangle_{B_{2R}(y)}-f(y)|
\end{equation*}
We can estimate the first and the third term of right hand-side of the above inequality using Lemma \ref{giusti}. Indeed, getting the limit in \eqref{giustieq} as $R'\rightarrow 0$ and writing $2R$ instead of $R$, for any $x\in\Omega$ we get
\begin{equation}\label{giusti3}
|\langle \tilde{f} \rangle_{B_{2R}(x)}-f(x)| \, \leq \, c\,[f]_{p,sp} |B_{2R}(x)|^{(sp-n)/np} \, \leq \, C [f]_{p,sp}\, R^{(sp-n)/p}
\end{equation}
where the constant $C$ is given by $c \,2^{(sp-n)/p}/ |B_1|$.
\vspace{2mm}

On the other hand, 
$$
|\langle \tilde{f} \rangle_{B_{2R}(x)}-\langle \tilde{f} \rangle_{B_{2R}(y)}|\leq \,|f(z)- \langle \tilde{f} \rangle_{B_{2R}(x)}|+|\tilde{f}(z)- \langle \tilde{f} \rangle_{B_{2R}(y)}|
$$
and so, integrating on $z\in B_{2R}(x)\cap B_{2R}(y)$, we have
\begin{eqnarray*}
&& |B_{2R}(x)\cap B_{2R}(y)| \, |\langle \tilde{f}\rangle_{B_{2R}(x)}-\langle \tilde{f}\rangle_{B_{2R}(y)}| \\[1ex]
&& \qquad \quad \leq \,  \int_{B_{2R}(x)\cap B_{2R}(y)} |\tilde{f}(z)-\langle \tilde{f} \rangle_{B_{2R}(x)}|\,dz \\
&& \qquad \quad \quad + \int_{B_{2R}(x)\cap B_{2R}(y)} |\tilde{f}(z)-\langle \tilde{f}\rangle_{B_{2R}(y)}|\,dz \\[1ex]
&& \qquad \quad \leq \, \int_{B_{2R}(x)} |\tilde{f}(z)-\langle \tilde{f}\rangle_{B_{2R}(x)}|\,dz+ \int_{ B_{2R}(y)} |\tilde{f}(z)-\langle \tilde{f}\rangle_{B_{2R}(y)}|\,dz.
\end{eqnarray*}
Furthermore, since $B_R(x)\cup B_R(y)\subset \big(B_{2R}(x)\cap B_{2R}(y)\big) $, we have 
$$
|B_R(x)|\leq |B_{2R}(x)\cap B_{2R}(y)| \quad \mbox{and}\quad |B_R(y)|\leq |B_{2R}(x)\cap B_{2R}(y)|
$$
and so
\begin{eqnarray*}
|\langle \tilde{f}\rangle_{B_{2R}(x)}- \langle \tilde{f}\rangle_{B_{2R}(y)}| & \leq &\frac{1}{|B_R(x)|} \int_{B_{2R}(x)} |\tilde{f}(z)-\langle \tilde{f}\rangle_{B_{2R}(x)}|\,dz\\[1ex]
&&  + \,\frac{1}{|B_R(y)|} \int_{ B_{2R}(y)} |\tilde{f}(z)-\langle \tilde{f}\rangle_{B_{2R}(y)}|\,dz.
\end{eqnarray*}
An application of the H\"older inequality gives
\begin{eqnarray}\label{giusti5}
&& \frac{1}{|B_R(x)|} \int_{B_{2R}(x)} |\tilde{f}(z)-\langle \tilde{f}\rangle_{B_{2R}(x)}|\,dz \!\! \nonumber \\
&& \qquad \qquad \qquad \quad \leq \, \frac{|B_{2R}(x)|^{{(p-1)}/{p}}}{|B_R(x)|}\left( \int_{B_{2R}(x)} |\tilde{f}(z)-\langle \tilde{f}\rangle_{B_{2R}(x)}|^p\,dz \right)^{\!\!1/p} \nonumber \\[1ex]
&& \qquad \qquad \qquad \quad \leq \, \frac{|B_{2R}(x)|^{{(p-1)}/{p}}}{|B_R(x)|} (2R)^{s} [f]_{p,sp} \nonumber \\[1ex]
& & \qquad \qquad \qquad \quad \leq \, C \,[f]_{p,sp}\, R^{(sp-n)/p}\,.
\end{eqnarray}
Analogously, we obtain
\begin{equation}\label{giusti6}
\frac{1}{|B_R(y)|} \int_{B_{2R}(y)} |\tilde{f}(z)-\langle \tilde{f} \rangle_{B_{2R}(y)}|\,dz  \leq  C \,[f]_{p,sp}\, R^{(sp-n)/p}\,.
\end{equation}
\vspace{2mm}

Combining \eqref{giusti3}, \eqref{giusti5} with~\eqref{giusti6} it follows
\begin{equation}\label{giusti8}
|f(x)-f(y)|\leq C\, [f]_{p,sp}\, |x-y|^{(sp-n)/p},
\end{equation}
up to relabeling the constant $C$.
\vspace{2mm}

Therefore, by taking into account~\eqref{campan1}, we can conclude that $f\in C^{0,\alpha}(\Omega)$, with $\alpha= (sp-n)/p$. 

Finally, taking $R_0<\rm{diam}(\Omega)$ (note that the latter can be possibly infinity), using estimate in~\eqref{giusti3} and the H\"older inequality we have,  for any $x\in \Omega$,
\begin{eqnarray}\label{giusti9}
|f(x)|&\leq &|\langle \tilde{f}\rangle_{B_{R_0}(x)}|+|f(x)- \langle \tilde{f}\rangle_{B_{R_0}(x)}| \nonumber \\
&\leq & \frac{C}{|B_{R_0}(x)|^{1/p}} \|f\|_{L^p(\Omega)} + c\,[f]_{p,sp} \,|B_{R_0}(x)|^{\alpha}.
\end{eqnarray}
Hence, by~\eqref{campan1}, \eqref{giusti8} and~\eqref{giusti9}, we get
\begin{eqnarray*}
\|f\|_{C^{0,\alpha}(\Omega)}&\,=\,&\|f\|_{L^{\infty}(\Omega)}+ \sup_{{x,y\in\Omega}\atop{x\neq y}}
\frac{|f(x)-f(y)|}{ |x-y|^{\alpha}}\\
&\,\leq\,& C \left( \,\|f\|_{L^p(\Omega)} + [f]_{p,sp} \, \right)\\
\\
&\,\leq\,& C  \|f\|_{W^{s,p}(\Omega)} \,.
\end{eqnarray*}
for a suitable positive constant $C$.
\end{proof}

\begin{rem} 
The estimate in~\eqref{campana} says that~$f$ belongs to the Campanato 
space~${\LL}^{p,\lambda}$, with~$\lambda:=sp$, (see~\cite{Cam63} and, e.g.,~\cite[Definition~2.4]{Giu94}).
Then, the conclusion in the proof of~Theorem~\ref{thm_holder} is actually an application of  the Campanato
Isomorphism (see, for instance,~\cite[Theorem~2.9]{Giu94}). 
\end{rem}

\vspace{2mm}

Just for a matter of curiosity, we observe that, according to the definition~\eqref{def1}, the fractional Sobolev space $W^{s,\infty}(\Omega)$ could be view as the space of functions
\begin{equation*}
\left\{ u\in L^\infty(\Omega)\; :\; \frac{|u(x)-u(y)|}{|x-y|^{s}} \in L^\infty(\Omega \times \Omega)  \right\},
\end{equation*}
but this space just boils down to $C^{0,s}(\Omega)$, that is consistent with the H\"older embedding proved in this section; i.e., taking formally $p=\infty$ in~Theorem~\ref{thm_holder}, the function $u$ belongs to $C^{0,s}(\Omega)$. 

\vspace{3mm}

\section{Some counterexamples in non-Lipschitz domains}\label{sec_example}

When the domain $\Omega$ is not Lipschitz, some interesting things happen, as next examples show.

\vspace{2mm}

\begin{exe}\label{exa_lip}
{\em Let~$s\in(0,1)$. We will construct a function $u$ in $W^{1,p}(\Omega)$ that does not belong to $W^{s,p}(\Omega)$, providing a counterexample to~Proposition~\ref{enrico23} when the domain is not Lipschitz.}

\vspace{2mm}

Take any
\begin{equation}\label{CON.1}
p\in (1/s,+\infty).\end{equation}
Due to~\eqref{CON.1}, we can fix
\begin{equation}\label{CON.2}
\kappa >\frac{p+1}{sp-1}.
\end{equation}
We remark that~$\kappa>1$.
\vspace{2mm}

Let us consider the cusp in the plane
$$
\mathcal{C}:=\big\{ (x_1,x_2)
{\mbox{ with }} x_1\le 0 {\mbox{ and }} |x_2|\le |x_1|^\kappa\big\}
$$
and take polar coordinates on~$\R^2\setminus \mathcal{C}$, say~$\rho=\rho(x)\in
(0,+\infty)$ and~$\theta=\theta(x)\in (-\pi,\pi)$, with~$x=(x_1,x_2)
\in\R^2\setminus \mathcal{C}$.
\vspace{2mm}

We define the function~$u(x):=\rho(x) \theta(x)$
and the heart-shaped domain~$\Omega:=(\R^2\setminus \mathcal{C})\cap B_1$, with $B_1$ being the unit ball centered in the origin.
Then, $u\in W^{1,p}(\Omega)\setminus W^{s,p}(\Omega)$.

\vspace{1mm}

To check this, we observe that
$$ \partial_{x_1} \rho \, = \,  (2\rho)^{-1} \partial_{x_1} \rho^2
 \, =  \, (2\rho)^{-1}\partial_{x_1}(x_1^2+x_2^2) \, = \, \frac{x_1}{\rho}$$
and, in the same way,
$$ \partial_{x_2}\rho=\frac{x_2}{\rho}.$$
Accordingly,
\begin{eqnarray*}
 1 \! & = & \! \partial_{x_1} x_1 \, = \, \partial_{x_1} (\rho\cos\theta)
 \, = \,  \partial_{x_1}\rho \cos\theta -\rho\sin\theta\partial_{x_1} \theta
 \, = \,  \frac{x_1^2}{\rho^2}-x_2\partial_{x_1}\theta \\
& = & \! 1-\frac{x_2^2}{\rho^2}-x_2\partial_{x_1}\theta.
\end{eqnarray*}
That is
$$ \partial_{x_1}\theta=-\frac{x_2}{\rho^2}.$$
By exchanging the roles of~$x_1$ and~$x_2$ (with some
care on the sign of the derivatives of the trigonometric
functions), one also obtains
$$ \partial_{x_2}\theta=\frac{x_1}{\rho^2}.$$
Therefore,
$$ \partial_{x_1} u \, = \,  \rho^{-1} (x_1\theta-x_2)\; \, 
{\mbox{ and }}\;  \, \partial_{x_2} u \, = \,  \rho^{-1} (x_2\theta+x_1)$$
and so
$$ |\nabla u|^2 = \theta^2 +1  \, \le \,  \pi^2+1.$$
This shows that~$u\in W^{1,p}(\Omega)$.
\vspace{2mm}

On the other hand, 
let us fix~$r\in(0,1)$,
to be taken arbitrarily small at the end, and
let us define~$r_0:=r$ and, for any~$j\in\N$, $r_{j+1}:=r_j-r_j^\kappa$.
By induction, one sees that~$r_j$ is strictly decreasing, that~$r_j>0$
and so~$r_j\in(0,r)\subset (0,1)$. 
Accordingly, we can define
$$ \ell:=\lim_{j\rightarrow+\infty} r_j\in [0,1].$$
By construction
$$ \ell \, = \,  \lim_{j\rightarrow+\infty} r_{j+1}
 \, = \, \lim_{j\rightarrow+\infty} r_j-r_j^\kappa
 \, = \, \ell -\ell^\kappa,$$
hence~$\ell=0$. As a consequence,
\begin{eqnarray}\label{CON.3}
\sum_{j=0}^{+\infty} r_j^\kappa \! & = & \! \lim_{N\rightarrow+\infty} 
\sum_{j=0}^{N} r_j^\kappa \, =\, \lim_{N\rightarrow+\infty}
\sum_{j=0}^{N} r_j-r_{j+1} \nonumber \\
&  = &  \! \lim_{N\rightarrow+\infty} r_0-r_{N+1} \, = \, r.
\end{eqnarray}
We define
\begin{eqnarray*}
{\mathcal{D}}_j& :=&
\big\{ (x,y)\in \R^2\times\R^2 {\mbox{ s.t. }} x_1,y_1\in (-r_j,-r_{j+1}),\\
&&\qquad\qquad\quad
x_2\in (|x_1|^\kappa, 2|x_1|^\kappa)
{\mbox{ and }}
-y_2\in (|y_1|^\kappa, 2|y_1|^\kappa)
\big\}.
\end{eqnarray*}
We observe that
\begin{eqnarray*}
 \Omega\times\Omega  &\supseteq &
\big\{ (x,y)\in\R^2\times\R^2 {\mbox{ s.t. }} x_1,y_1\in (-r,0),\\
&&\qquad\quad
x_2\in (|x_1|^\kappa, 2|x_1|^\kappa)
{\mbox{ and }}
-y_2\in (|y_1|^\kappa, 2|y_1|^\kappa)
\big\}\\
&\supseteq& \bigcup_{j=0}^{+\infty} {\mathcal{D}}_j,
\end{eqnarray*}
and the union is disjoint.
Also,
$$r_{j+1} \, = \, r_j( 1-r_j^{\kappa-1}) \, \ge \,  r_j (1-r^{\kappa-1}) \, \ge \,  \frac{r_j}2,$$
for small~$r$. Hence,
if~$(x,y)\in {\mathcal{D}}_j$,
$$ |x_1|\le r_j\le 2r_{j+1}\le 2|y_1|$$
and, analogously,
$$ |y_1|\le 2|x_1|.$$
Moreover,
if~$(x,y)\in {\mathcal{D}}_j$,
$$ |x_1-y_1| \, \le \,  r_j-r_{j+1} \, = \, r_j^\kappa
 \, \le \,  2^\kappa r_{j+1}^\kappa \, \le \,  2^\kappa |x_1|^\kappa$$
and
$$ |x_2-y_2| \, \le \,  |x_2|+|y_2| \, \le \,  2|x_1|^\kappa+2|y_1|^\kappa
 \, \le \,  2^{\kappa+2} |x_1|^\kappa.$$
As a consequence, if~$(x,y)\in {\mathcal{D}}_j$,
$$ |x-y|\le 2^{\kappa+3} |x_1|^\kappa.$$
Notice also that, when~$(x,y)\in{\mathcal{D}}_j$,
we have~$\theta(x)\ge \pi/2$ and $\theta(y)\le -\pi/2$,
so
$$ u(x)-u(y) \, \ge \,  u(x) \, \ge \,  \frac{\pi \,\rho(x)}{2} \, \ge \, 
\frac{\pi \,|x_1|}{2}.$$
As a consequence, for any~$(x,y)\in {\mathcal{D}}_j$,
$$
\frac{|u(x)-u(y)|^p}{|x-y|^{2+sp}}
\ge c |x_1|^{p-\kappa (2+sp)},
$$
for some~$c>0$.
Therefore,
\begin{eqnarray*}
&& 
\iint_{\mathcal{D}_j}\frac{|u(x)-u(y)|^p}{|x-y|^{2+sp}}\,dx\,dy\,
\ge\,
\iint_{ {\mathcal{D}}_j } {c}{|x_1|^{p-\kappa (2+sp)}}\,dx\,dy
\\
&&\qquad\qquad
=c \int_{-r_j}^{-r_{j+1}} \,dx_1\int_{-r_j}^{-r_{j+1}}\,dy_1
\int_{|x_1|^\kappa}^{2|x_1|^\kappa}\,dx_2
\int_{-2|y_1|^\kappa}^{-|y_1|^\kappa}\,dy_2 {|x_1|^{p-\kappa (2+sp)}}
\\
&&\qquad\qquad
= c \int_{-r_j}^{-r_{j+1}} \,dx_1\int_{-r_j}^{-r_{j+1}}\,dy_1 
{|x_1|^{p-\kappa (2+sp)}} |x_1|^\kappa |y_1|^\kappa
\\ &&\qquad\qquad
\ge c \, 2^{-\kappa} \int_{-r_j}^{-r_{j+1}} 
\,dx_1\int_{-r_j}^{-r_{j+1}}\,dy_1
{|x_1|^{p-\kappa sp}}
\\
&&\qquad\qquad
\ge  c\, 2^{-\kappa} \int_{-r_j}^{-r_{j+1}}
\,dx_1\int_{-r_j}^{-r_{j+1}}\,dy_1
{r_j^{p-\kappa sp}} 
\\ &&\qquad \qquad = \, c\, 2^{-\kappa} 
{r_j^{p-\kappa sp+2\kappa}} \, = \, c\, 2^{-\kappa}
{r_j^{\kappa-\alpha}},
\end{eqnarray*}
with
\begin{equation}\label{CON.4}
\alpha:=\kappa(sp-1)-p>1,
\end{equation}
thanks to~\eqref{CON.2}.

\vspace{1mm}

In particular,
$$\iint_{\mathcal{D}_j}\frac{|u(x)-u(y)|^p}{|x-y|^{2+sp}}\,dx\,dy
\, \ge \, c\, 2^{-\kappa} r^{-\alpha}
{r_j^{\kappa}}$$
and so, by summing up and exploiting~\eqref{CON.3},
$$
\int_\Omega\int_{\Omega}\frac{|u(x)-u(y)|^p}{|x-y|^{2+sp}}\,dx\,dy
\, \ge \, \sum_{j=0}^{+\infty}
\iint_{\mathcal{D}_j}\frac{|u(x)-u(y)|^p}{|x-y|^{2+sp}}\,dx\,dy
\, \ge \, c\, 2^{-\kappa} r^{1-\alpha}.$$
By taking~$r$ as small as we wish and recalling~\eqref{CON.4},
we obtain that
$$ \int_\Omega\int_{\Omega}\frac{|u(x)-u(y)|^p}{|x-y|^{2+sp}}\,dx\,dy
=+\infty,$$
so~$u\not\in W^{s,p}(\Omega)$.\hfill $\Box$
\end{exe}

\vspace{2mm}

\begin{exe} \label{exe_ovidiu}
{\em Let~$s\in(0,1)$. We will construct a sequence of functions $\{f_n\}$ bounded in $W^{s,p}(\Omega)$ that does not admit any convergent subsequence in~$L^q(\Omega)$, providing a counterexample to Theorem~\ref{thm_comp} when the domain is not Lipschitz.}
\vspace{2mm}

We follow an observation by~\cite{Sav11}. For the sake of simplicity, fix $n=p=q=2$. We take $a_k:={1}/{C^k}$ for a constant $C>10$ and we consider the set $\Omega= \bigcup_{k=1}^\infty B_k$ where, for any $k\in \N$, $B_k$ denotes the ball of radius $a^2_k$ centered in $a_k$. Notice that
$$
\dys 
a_k \to 0 \ \text{as} \ k\to\infty \ \ \text{and} \  \
a_k-a_k^2>a_{k+1}+a_{k+1}^2.
$$
Thus, $\Omega$ is the union of disjoint balls, it is bounded and it is not a Lipschitz domain.

For any $n\in \N$, we define the function $f_n:\Omega\to \R$ as follows
$$
\dys
f_n(x)=
\begin{cases}
\pi^{-{\frac{1}{2}}} \,a_n^{-2} & x\in B_n, \\
\,\,\;0 & x\in \Omega\setminus B_n\,.
\end{cases}
$$
We observe that we cannot extract any subsequence convergent in $L^2(\Omega)$ from the sequence of functions $\{f_n\}$, because $f_n(x)\rightarrow 0$ as $n\rightarrow +\infty$, for any fixed $x\in\Omega$ but
$$
\|f_n\|^2_{L^2(\Omega)}\, =\, \int_{\Omega} |f_n(x)|^2 \,dx \, = \, \int_{B_n} \pi^{-1}\,a_n^{-4}\,dx  \, = \, 1.
$$
\vspace{1mm}

Now, we compute the $H^s$~norm of~$f_n$ in~$\Omega$. We have
\begin{eqnarray}\label{sara}
\!\!\!\! \int_{\Omega} \int_{\Omega} \frac{|f_n(x)-f_n(y)|^2}{|x-y|^{2+2s}}\,dx\,dy
\!\! &=& \!\!
2 \int_{\Omega\setminus B_n}\!\int_{B_n} \frac{\pi^{-1} a_n^{-4}}{|x-y|^{2+2s}}\,dx\,dy \nonumber \\
&=& \!\!  2\pi^{-1}\sum_{k\neq n} \int_{B_k}\!\int_{B_n} \frac{ a_n^{-4}}{|x-y|^{2+2s}}\,dx\,dy.
\end{eqnarray}

Thanks to the choice of $\{a_k\}$ we have that
$$
|a_n^2+a_k^2|=a_n^2+a_k^2\leq \frac{|a_n-a_k|}{2}.
$$ 
Thus, since $x\in B_n$, $y\in B_k$, it follows 
\begin{eqnarray*}
|x-y| \! &\geq &\! |a_n-a_n^2-(a_k+a_k^2)|\,= \,|a_n-a_k-(a_n^2+a_k^2)| \\
&\geq &\! |a_n-a_k|-|a_n^2+a_k^2|\, \geq \, |a_n-a_k|-\frac{|a_n-a_k|}{2}\\
&=& \!\frac{|a_n-a_k|}{2}.
\end{eqnarray*}
Therefore, 
\begin{eqnarray}\label{sara2}
\int_{B_k}\int_{B_n} \frac{ a_n^{-4}}{|x-y|^{2+2s}}\,dx\,dy 
\! & \leq & \! 2^{2+2s} \int_{B_k}\int_{B_n} \frac{ a_n^{-4}}{|a_n-a_k|^{2+2s}}\,dx\,dy \nonumber\\
& = & \! 2^{2+2s} \pi^2 \frac{a_k^4}{|a_n-a_k|^{2+2s}}.
\end{eqnarray}
Also, if $m\geq j+1$ we have
\begin{equation}\label{sara28}
a_j-a_m\geq a_j-a_{j+1}=\frac{1}{C^j}-\frac{1}{C^{j+1}} =\frac{1}{C^j}\left( 1-\frac{1}{C} \right)\geq \frac{a_j}{2}.
\end{equation}
Therefore, combining \eqref{sara28} with \eqref{sara} and \eqref{sara2}, we get 
\begin{eqnarray*}
&& \int_{\Omega} \int_{\Omega} \frac{|f_n(x)-f_n(y)|^2}{|x-y|^{2+2s}}\,dx\,dy \\
&& \qquad \qquad \qquad \quad \leq \, 2^{3+2s}\pi\sum_{k\neq n} \frac{a_k^4}{|a_n-a_k|^{2+2s}} \\
\\
 &&  \qquad \qquad \qquad \quad = \, 2^{3+2s}\pi\left(\sum_{k< n} \frac{a_k^4}{(a_k-a_n)^{2+2s}} + \sum_{k>n} \frac{a_k^4}{(a_n-a_k)^{2+2s}}\right) \\
\\
 && \qquad \qquad \qquad \quad \leq  \, 2^{5+4s}\pi\left( \sum_{k< n} \frac{a_k^4}{a_k^{2+2s}}+ \sum_{k> n} \frac{a_k^4}{a_n^{2+2s}} \right)\\
 \\
 && \qquad \qquad \qquad \quad \leq  \, 2^{6+4s}\pi\sum_{k\neq n} a_k^{2-2s}
 \, =  \, 2^{6+4s}\pi\sum_{k\neq n} \left( \frac{1}{C^{2-2s}} \right)^k \, < \, +\infty.
\end{eqnarray*}
This shows that $\{f_n\}$ is bounded in $H^s(\Omega)$.
\hfill $\Box$
\end{exe}

\end{document}